\documentclass{amsart}
\usepackage{graphicx}
\usepackage{amssymb}
\usepackage{amscd}
\usepackage{latexsym}
\usepackage{amsfonts}
\usepackage{ragged2e}
\usepackage{amsmath}
\usepackage{float}
\usepackage[latin1]{inputenc}
\usepackage[english]{babel}

\newtheorem{prop}{Proposition}
\newtheorem{lemma}{Lemma}
\newtheorem{definition}{Definition}
\newtheorem{thm}{Theorem}

\newcommand{\eproof}{\begin{flushright} $\square$ \end{flushright}}

\let\eproof\endproof 

\usepackage[draft]{fixme}

\newcommand{\tr}{\mathop{\fam0 Tr}\nolimits}

\newcommand{\ch}{\mathop{\fam0 ch}\nolimits}

\newcommand{\Lie}[1]{\mbox{$\mathfrak #1$}}

\newcommand{\Td}{\mathop{\fam0 Td}\nolimits}

\newcommand{\C}{C}

\newcommand{\bX}{{\mathbf X}}
\newcommand{\Z}{{\mathbb Z}}

\newcommand{\ra}{\mathop{\fam0 \rightarrow}\nolimits}

\renewcommand{\epsilon}{\varepsilon}

\newcommand{\frakg}{\mathfrak{g}}

\newcommand{\bbC}{\mathbb{C}}

\newcommand{\ind}{\text{ind}}

\newcommand{\calE}{\mathcal{E}}

\DeclareMathOperator{\Tr}{Tr}

\DeclareMathOperator{\rk}{rk}

\def\XXint#1#2#3{{\setbox0=\hbox{$#1{#2#3}{\int}$}
\vcenter{\hbox{$#2#3$}}\kern-.5\wd0}}

\newcommand{\Sigmao}{\Sigma^\circ}

\def\tilde{\widetilde}

\def\bar{\overline}

\def\be{\begin{eqnarray}}
\def\ee{\end{eqnarray}}

\def\CN{{\mathcal N}}

\def\CL{{\mathcal L}}

\def\tr{{\mathrm{tr}\,}}
\def\ch{{\mathrm{ch}}}
\def\Td{{\mathrm{Td}}}

\renewcommand{\C}{{\mathbb C}}

\def\ind{\mathrm{Ind}}
\def\^{{\wedge}}
\newcommand{\cp}{{\mathbb{P}}}

\begin{document}

\title[The Verlinde formula for Higgs bundles]{The Verlinde formula for\\ Higgs bundles}

\date{\today}
\author{Jørgen Ellegaard Andersen}
\address{Center for Quantum Geometry of Moduli
  Spaces\\ 
  Department of Mathematics\\
  University of Aarhus\\
  DK-8000, Denmark}
	\email{andersen@qgm.au.dk}
  \author{Sergei Gukov}
\address{Walter Burke Institute for Theoretical Physics\\
 California Institute of Technology\\
Pasadena, CA 91125, USA\\
\newline \phantom{bb}
Max-Planck-Institut f\"ur Mathematik, \\ Vivatsgasse 7, D-53111 Bonn, Germany
}
\email{gukov@theory.caltech.edu}
\author{Du Pei}
\address{Walter Burke Institute for Theoretical Physics\\
 California Institute of Technology\\
Pasadena, CA 91125, USA\\ \newline \phantom{bb}
Center for Quantum Geometry of Moduli
  Spaces\\ 
  Deparment of Mathematics\\
  University of Aarhus\\
  DK-8000, Denmark}
\email{pei@caltech.edu}

\begin{abstract}
We propose and prove the Verlinde formula for the quantization of the Higgs bundle moduli spaces and stacks for any simple and simply-connected reductive group $G$. This generalizes the equivariant Verlinde formula for the case of $G={\rm SL}(n,{\mathbb C})$ proposed in \cite{GP}. We further establish a Verlinde formula for the quantization of parabolic Higgs bundle moduli spaces and stacks. Finally, we prove that these dimensions form a one-parameter family of $1+1$-dimensional TQFT, uniquely classified by the complex Verlinde algebra, which is a one-parameter family of Frobenius algebras. We construct this one-parameter family of Frobenius algebras as a deformation of the classical Verlinde algebra for $G$.
\end{abstract}

\thanks{Supported in part by the center of excellence grant ``Center for Quantum Geometry of Moduli Space" from the Danish National Research Foundation (DNRF95), by the Walter Burke Institute for Theoretical Physics, and by the U.S.~Department of Energy, Office of Science, Office of High Energy Physics, under Award Number DE-SC0011632.}

\maketitle

\section{Introduction}

Let $G$ be a simple and simply-connected reductive group over the complex numbers and $K$ its compact form.
Let $M$ be the moduli space of semi-stable holomorphic $G$-bundles on $\Sigma$, where $\Sigma$ is a smooth algebraic curve over the complex numbers. Let $L$ be the determinant bundle over $M$.
In order to recall the Verlinde formula for the dimension of the space of holomorphic sections of $L^k$ over $M$ for any non-negative integer level $k$, we specify some notations.

Let $T\subset G$ be the maximal torus and $W$ the Weyl group. 
We now consider the killing form $\langle\cdot,\cdot\rangle$ on ${\mathfrak g} = \mathrm{Lie}(G)$, normalized so that the longest root has norm squared equal to $2$. Let ${\mathfrak t} = \mathrm{Lie}(T)$. Pick a level $k\in {\mathbb Z}_+$ and consider the map
$$i_k : {\mathfrak t} \ra {\mathfrak t}^*,$$
given by
$$ i_k(\xi)(\eta) = (k+h)\langle\xi,\eta\rangle,$$
where $h$ is the dual Coxeter number. This descends to a map 
$$ \chi : T \ra T^{*}.$$
Let 
$$ F = \ker \chi.$$
Let $ \rho = \frac12 \sum_{\alpha\in \frak{R}_+} \alpha$, where $\frak{R}_+$ is the set of positive roots of $\frakg$.
Put
$$ F_\rho = \chi^{-1}(e^{2\pi i \rho}).$$ As usual we denote by
$$\Delta = \prod_{\alpha\in \frak{R}_+} 2 \sin\left(\frac{i\alpha}{2}\right)$$
the Weyl denominator. Define the following function
$$ \theta(f) = \frac{\Delta(f)^2}{|F|}$$
for $f\in T$. 

\begin{thm}[Szenes, Donaldson, Bertram-Szenes, Daskalopoulos-Wentworth,\\ Thaddeus, Beauville-Laszlo, Kumar-Narasimhan-Ramanathan, Pauly,  Faltings, \\ Teleman, Teleman-Woodward]\label{thm:CVF} 
\mbox{ }\\ For all non-negative integers $k$ we have that
$$\dim H^0(M, L^k) = \sum_{f\in F^{\rm reg}_{\rho} / W} \theta(f)^{1-g}.$$
\end{thm}

The earliest proof of this theorem are in the special cases of $G=\rm{SL}(2,{\mathbb C})$ \cite{Sz1, D, Sz2, BS, Th, DW1,DW2}, where \cite{DW1,DW2} also coveres the parabolic case. A different kind of proof is found in \cite{BL, KNR}, which established an isomorphism  between this space of holomorphic sections and the space of conformal blocks, whose dimension had already been proven to be given by the Verlinde formula in \cite{TUY}. This work was then extended to the parabolic case in \cite{CP} covering $G= \rm{SL}(n,{\mathbb C})$ case and in \cite{Te1/2} covering the generality stated here and further in \cite{F} to an even more general setting. Finally the index theorem of Teleman and Woodward \cite{TW} gives a direct index computation proof.

In order to generalize this formula to the case of the Higgs bundle moduli spaces, we consider the ${\mathbb C}^*$-action this moduli space has, which provides a grading on the infinite-dimensional quantization of this moduli space.

Let $M_H$ be the moduli space of semi-stable $G$-Higgs bundles on $\Sigma$.  By abuse of notation, we also let $L$ be the determinant line bundle over $M_H$.
 The ${\mathbb C}^*$-action on $M_H$ naturally lifts to a ${\mathbb C}^*$-action on $L$. We therefore get a decomposition
$$ H^0(M_H, L^k) \cong \bigoplus_{n=0}^\infty  H_n^0(M_H, L^k),$$
into finite-dimensional subspaces, where ${\mathbb C}^*$ acts on the $n$'th one by the character of the $n$ tensor power of the identity representation of ${\mathbb C}^*$ on ${\mathbb C}$.
We can therefore define the formal sum in the variable $t$ by
$$ \dim_t H^0(M_H, L^k) = \sum_{n=0}^\infty \dim H_n^0(M_H, L^k) t^n$$
and we denote it the {\em Hitchin character}.

Following now \cite{TW}, consider the following deformation 
$$\chi_t = \chi \prod_{\alpha\in \frak{R}_+} \left(\frac{1-te^\alpha}{1-te^{-\alpha}}\right)^\alpha : T \ra T^*.$$
Let
$$F_{\rho,t} = \chi_t^{-1} (e^{2\pi i \rho})\subset T.$$
Further consider the following function defined on $T$
$$\theta_t = (1-t)^{\mathrm{rk} G}\frac{1}{|F|} \frac{ \prod_{\alpha}(1-e^\alpha)(1-te^\alpha)}{\det H_t^\dagger},$$
where the product is taken over all of the roots of $\frak{g}$ and $H_t^\dagger$ is the endomorphism of the Cartan subalgebra $\Lie{t}$ of $\frak{g}$ corresponding to the Hessian $H_t$ of the function 
\be\label{FunD}
D_t(\xi) = \frac12(k+h)\langle\xi,\xi\rangle - \tr_\frakg ({\rm Li}_2(te^\xi))+\langle\rho,\xi\rangle.
\ee

In this paper we prove the following {\em Verlinde formula for Higgs bundles}.

\begin{thm}\label{thm:Main1}
For all $g>1$ and all non-negative integers $k$ we have that
\be \label{CVF}
\dim _t H^0(M_H, L^k) = \sum_{f\in F^{\rm reg}_{\rho,t} / W} \theta_t(f)^{1-g}.
\ee

\end{thm}

We want to stress that our proof relies strongly on the index formula of Teleman and Woodward \cite{TW} and on a series of beautiful results by Teleman that provide deep understanding of natural sheaf cohomology groups over the moduli spaces and stacks of $G$-bundles on curves \cite{Te0, Te1, Te3,Te}, as well as further developments by him with Fishel and Grojnowski \cite{FGT} and with Frenkel \cite{FT}. 

When $g>1$ a codimension argument shows that
\be\label{SpaceStack}
H^0(M_H, L^k)  = H^0({\mathfrak M}_H, {\mathfrak L}^k),\ee
where ${\mathfrak M}_H$ is the stack of $G$-Higgs bundles and we denoted the determinant line bundle over the stack by ${\mathfrak L}$. In fact, it is for the stack that the Verlinde formula (\ref{CVF}) is deduced and then the above identification of vector spaces gives the theorem. We remark that the right hand side of \eqref{CVF} is, in fact,  for all genus $g$, equal to the index of ${\mathfrak L}^k$ over the stack ${\mathfrak M}_H$ (see Theorem \ref{thm:Main5} below), but for genus zero and one, $\frak{M}_H$ is a derived stack and \eqref{SpaceStack} is no longer valid. The proof of Theorem \ref{thm:Main1} is presented in Section \ref{Proof}, where we also establish the Verlinde formula for the $R$-twisted Higgs bundles. Please see Section \ref{Proof} for definitions and details.

Let us now consider the parabolic generalization. Let $x_1,\ldots, x_m$ be distinct points on $\Sigma$ and $D=\sum_i x_i$ be a divisor on $\Sigma$. 
Let $\lambda$ be a level-$k$ integrable dominant weight for $G$, {\it e.g.}
$$\langle \lambda , \vartheta \rangle \leq k,$$
where $\vartheta$ is the highest root of $G$. We denote by $\Lambda_k$ the set of these level-$k$ integrable dominant weights.

Let $P$ be the moduli space of parabolic $G$-bundles for a choice of parabolic structures, determined by a parabolic sub-group $P_i\subset G$ at $x_i$, together with a choice of a weights $\lambda_i\in \Lambda_k$, $i=1,\ldots,m$. We let $P_H$ be the corresponding moduli space of parabolic $G$-Higgs bundles  with respect to the same weight vectors.
Let $L_k$ be the level-$k$ determinant line bundle over $P$ and $P_H$ determined by the weights $\vec \lambda = (\lambda_1, \ldots, \lambda_n)$ assigned to each of the marked point $(x_1, \ldots, x_n)$ respectively. The compatibility condition between $\lambda_i$ and $P_i$ we must require, is that $P_i$ leaves $\lambda_i$ invariant, for the line bundle $L_k$ to exist.\footnote{Our convention will be the following. Choose a Borel subgroup $B$ of $G$, whose Lie algebra is $\frak{b}$. Then $\frak{g}/\frak{b}$ will be spanned by the \emph{positive} roots, and a weight $\lambda$ annihilated by $B$ under the right action will live in the dominant Weyl chamber. As $P_i\supset B$, its Lie algebra $\frak{p}_i=\mathrm{Lie}(P_i)$ will also contain all negative roots.} Again, the natural ${\mathbb C}^*$-action on $P_H$ then lifts to $L_{k}$ and we have that
$$ H^0(P_H, L_{k}) \cong \bigoplus_{n=0}^\infty  H_n^0(P_H, L_{k})$$
and we can again define the dimension, or the {\em parabolic Hitchin character}, to be the formal generating series, in a variable $t$ say, for the dimensions of these finite-dimensional ${\mathbb C}^*$-weight spaces as in the non-parabolic case. We obtain the following {\em Verlinde formula for parabolic Higgs bundles}.

\begin{thm}\label{thm:Main2} 
For $g>1$ we have that
\be\label{CVFP}
\dim_t H^0(P_H, L_{k})  =\sum_{f\in F^{\rm reg}_{\rho,t} / W} \theta_t(f)^{1-g} \prod_i\Theta_{\lambda_i,P_i,t}(f), 
\ee
where $\Theta_{\lambda,P',t}$ is defined for a pair $(\lambda,P')$ of a level-$k$ highest weight $\lambda$ of $G$ and a compatible parabolic subgroup $P'$ as follows
$$
\Theta_{\lambda, P',t} = \sum_{w\in W} \frac{(-1)^{l(w)} e^{w(\lambda+\rho)}}{\Delta \prod_{\alpha\in \frak{R}(\frak{g}/\frak{p}')}(1-te^{w(\alpha)})}
$$
and $\frak{R}(\frak{g}/\frak{p}')$ is the subset of the positive roots of $\frak{g}$ invariant under $P'$.
\end{thm}

The same remarks concerning the proof and its $R$-twisted generalizations of Theorem \ref{thm:Main1} also applies to the proof of this Theorem and its $R$-twisted generalizations. Please see Section \ref{ParaProof}.

We further can show that these dimension formulae have the following $1+1$-dimensional TQFT properties (in fact, it is the indices over the \emph{stacks} that truly has the TQFT behaviour, reaffirming the advantage of favouring the stack over the moduli space). First we introduce the TQFT.

\begin{definition}
For any $g,n,k$ non-negative integers and any $n$-tuple of level-$k$ integrable dominant weights, we define $d_{n,g}(\vec \lambda)(t)$ for sufficient small $t$ by
$$ d_{g,n}(\vec \lambda)(t) = \sum_{f\in F^{\rm reg}_{\rho,t} / W} \theta_t(f)^{1-g} \prod_i\Theta_{\lambda_i,P_i,t}(f).$$
\end{definition}

Thus $d_{g,n}(\vec \lambda) $ is defined to be the right-hand side of (\ref{CVFP}), which is also equal to the index of the $k$-th power of the determinant bundle over the moduli stack of $G$-Higgs bundles as we will see in Section \ref{tqft}, where we will also argue that
$$
d_\lambda := d_{0,2}^{-1}(\lambda, \lambda^*),
$$
takes the simple form of $d_\lambda(t) = P_{t^{1/2}}(BH_\lambda)$ with $P_{\bullet}$ denoting the Poincare polynomial and $H_\lambda$ being the stabilizer of $e^{\lambda^\vee/k}$ in $K$. We further observe that for small enough $t$ (uniform in $g,n, \vec \lambda$) the series $d_{g,n}(\vec \lambda)(t)$ are actually a holomorphic functions of $t$.
\begin{thm}\label{thm:Main3}
Assume that all $P_i$'s used are the Borel subgroup $B$ of $G$. For any $g, g_1, g_2$ and $n,n_1,n_2$ non-negative integers and all vectors $\vec \lambda, \vec \lambda_1, \vec\lambda_2$ of level-$k$ integrable dominant weights, we have that 
$$
d_{g+1, n-2}(\vec \lambda) = \sum_{\lambda \in \Lambda_k} d_{g,n}(\vec \lambda, \lambda, \lambda^*) d_\lambda$$
and
$$d_{g_1+g_2,n_1+n_2}(\vec \lambda_1, \vec \lambda_2) = \sum_{\lambda \in \Lambda_k} d_{g_1,n_1+1}(\vec \lambda_1, \lambda) d_{g_2,n_2+1}(\vec \lambda_2, \lambda^*) d_\lambda.$$
\end{thm}

In Section \ref{tqft} we provide the proof of this Theorem. These TQFT rules of course allows us to define the {\em complex Verlinde algebra} associated to the theory as follows.

\begin{definition}\label{VA}
Let ${\mathcal V}_{G}^{(k)}$ be the free ${\mathbb C}$-vector space on the finite set $\Lambda_k$ with the following multiplication
$$ \lambda \star_t \mu = \sum_{\nu\in \Lambda_k} d_{0,3}(\lambda,\mu,\nu^*)(t)d_\nu(t) \nu,$$
and symmetric bilinear pairing
$$ \langle \lambda,\mu\rangle_t = \delta_{\lambda,\mu^*} d_{\lambda}(t)^{-1}$$
for small enough $t$.
\end{definition}

\begin{thm}\label{thm:Main4}
The triple $({\mathcal V}_{G}^{(k)}, \star_t, \langle \cdot,\cdot\rangle_t)$ is a Frobenius algebra for small enough $t$.
\end{thm}

The associativity of $\star_t$ and compatibility between $\star_t$ and $ \langle \cdot,\cdot\rangle_t$ follows immediately from Theorem \ref{thm:Main3} by the usual 2D TQFT classification construction, which we briefly recall in Section \ref{tqft}. We remark  that the complex Verlinde algebra $({\mathcal V}_{G}^{(k)}, \star_t, \langle \cdot,\cdot\rangle_t)$, e.g. this one parameter family of Frobenius algebras completely determines the one parameter family of TQFT $d_{g,n}(\vec \lambda)(t)$, for $t$ small.

Theorems \ref{thm:Main3} and \ref{thm:Main4} strongly suggest that there should be a kind of deformed Cohomological Field Theory associated to the corresponding bundles, by the pushforward from the universal $G$-Higgs bundle stack, which fibers over the moduli space  of curves. In fact, one can expect that they will satisfy nice factorization properties (indicated by the above recursion relations for the $d$'s) over the Deligne-Mumford compactifications of the moduli space of curves as was establish by Tsuchiya, Ueno and Yamada in \cite{TUY}, when one considers just moduli space of $G$-bundles. If this would be possible, most likely one can deform the constructions of \cite{AGO} and provides a kind of deformed CohFT description of the Chern classes of the resulting bundles and compute them by Topological Recursion.

From a physical point of view one should in fact expect that one can push these constructions all the way to the construction of a kind of deformed modular functor, as was done by Andersen and Ueno in \cite{AU1,AU2} in the case of $G$-bundles. In particular, this would require the construction of a projectively flat Hitchin or TUY connection in this Higgs bundle case. Pushing this further, one should be able to use this to provide a full geometric construction of Complex Quantum Chern-Simons theory for the group $G$ and possibly connect it with the theory construction by Andersen and Kashaev in \cite{AK1,AK2,AK3}. This would then require the development of an analog of the isomorphisms provided by Andersen and Ueno for $K=SU(n)$ in \cite{AU3,AU4}.

This paper is organized as follows.  We recall the Teleman-Woodward index formula in section \ref{TWIndex}. In Section \ref{Proof} we provide the proof of Theorem \ref{thm:Main1} concerning moduli spaces of Higgs bundles and we discuss the generalization to the $R$-twisted case. We give the proof of Theorem \ref{thm:Main2} regarding the parabolic case in Section \ref{ParaProof}, where we also discuss $R$-twisting. In particular, we also propose the relevant deformation of the usual character formulae relevant in this Higgs bundles case.
In Section \ref{Exceptional}, we discuss the special cases of $g$ being zero and one, together with some discussion of the rank one case in genus $2$.
In Section \ref{tqft}, we provide a 2D TQFT viewpoint on these formulae and prove Theorem \ref{thm:Main3}.

We would like to thank Indranil Biswas, Steve Bradlow, Peter Gothen, Tam\'as Hausel, Oscar Gargia-Prada, Nigel Hitchin, Tony Pantev, Brent Pym, Constantin Teleman, Richard Wentworth and Ke Ye for helpful discussions. Various stages of this work were performed while the authors were at the department of Mathematics and Physics at Caltech, at Centre for Quantum Geometry of Moduli Spaces, visiting the Chennai Mathematical Institute, and attending the Simons Summer Workshop 2016. We thank the hosts and organizers for their hospitality.

While the paper was being completed, we learned about the independent work from Daniel Halpern-Leistner on the same subject. After communicating with him, we decided to coordinate the arXiv submission. We would like to thank him for his kindness and understanding in this coordination process. Our work was motivated by the TQFT point of view, which drove us to considering parabolic Higgs bundles and to establish a gluing formula for these dimensions. On the other hand, Halpern-Leistner's paper \cite{HL} considered Higgs fields that are twisted by arbitrary line bundles for arbitrary reductive $G$ (with trivial fundamental group).

After the first version of the paper appeared on the arXiv, we learn from Tam\'as Hausel's talk at the conference ``Hitchin70: Differential Geometry and Quantization,'' held at QGM, Aarhus about his unpublished work with Andr\'as Szenes, in which the Verlinde formula for $SL(2,\C)$-Higgs bundles was obtained as a residue formula via direct computation.\footnote{This work was announced on September 18, 2003 at the geometry seminar at UT Austin.}

\section{The Teleman-Woodward Index formula}\label{TWIndex}

We now briefly review the results of \cite{TW}, setting up the stage for the proofs of the main theorems in later sections. An important lesson from \cite{TW} is that one should not work with the moduli space, but rather with the \emph{stack}, as the index theory on the latter behaves much more nicely. Therefore, let ${\mathfrak M}$ be the moduli stack of $G$-bundles over $\Sigma$. By abuse of notation, we also denote the determinant line bundle over ${\mathfrak M}$ by $L$.

The K$^0$-group of $\frak{M}$ is generated by the even Atiyah-Bott generators, which we now recall following \cite{AB} and \cite{TW}. Recall that the there is a universal principle $G$-bundle $E$ over $\frak{M}\times\Sigma$, and for a $G$ representation $V$, we denote the associated bundle $E(V)$. For a point $x\in\Sigma$, we denote by $E_x V$ the restriction of $E(V)$ to $\{x\}\times \frak{M}$, and $E_{\Sigma}V$ the total direct image of $E(V)\otimes K^{1/2}_\Sigma$ along the projection $\pi_{\frak{M}} : \frak{M}\times\Sigma \ra \frak{M}$. 

We have the following index theorem from \cite{TW}, which we only state for the even classes needed in this paper.

\begin{thm}[Teleman and Woodward]\label{TWInd}
The index formula for even classes is given by
$$\mathrm{Ind}\left(\frak{M};L^k\otimes \exp\left[s_1E_{\Sigma}V_1+\ldots+s_nE_{\Sigma}V_n\right]\otimes E_x V \right)=\sum_{f\in F^{\rm reg}_\rho /W}\theta_{\vec s}(f_{\vec s})^{1-g}\cdot \Tr_V\left(f_{\vec s}\right).$$
\end{thm}
Here ${\vec s}=(s_1,\ldots,s_n)$, while $\theta_{\vec{s}}$ and $f_{\vec{s}}$ is defined as follows. Consider the multi-parameter transformations on $G$ given by
$$g\mapsto m_{\vec s}(g)=g\cdot \exp\left[\sum_i s_i\nabla \Tr_{V_i}(g)\right], $$
with the gradient in the bilinear form given by $k+h$ times our normalized Killing form, which is used through the rest of this section as the metric on $T$. For small $\vec s$, there are unique solutions $f_{\vec s}\in T$, continuous in $\vec s$, to the equations 
\be\label{beq}
m_{\vec s}(f_{\vec s})=f
\ee
for $f\in F^{\rm reg}_\rho$. Let $H_{\vec s}$ be the Hessian of the function
$$\xi\mapsto \frac{k+h}{2}\langle\xi,\xi\rangle +  \sum_i s_i \Tr_{V_i}(e^\xi) + \langle\rho,\xi\rangle, $$
which can be converted into an endomorphism $H^\dagger_{\vec s}$ of $\frak{t}$. Now define  the function $\theta_{\vec s}$ on $T$ by
\be\label{thetas}
\theta_{\vec s}(f)=\det{}^{-1}\left[H_{\vec s}(f)^\dagger\right]\cdot \frac{\Delta(f)^2}{|F|}. 
\ee
Of particular interest to us is the series $\sum_p s_pE^*_{\Sigma}V_p$ given by 
$$ \lambda_{-t}(E^*_\Sigma \frak{g})=\exp\left[-\sum_{p>0}\frac{t^pE^*_{\Sigma}\psi^p(\frak{g})}{p}\right],  $$ 
where $\psi^p$ is the $p$'th Adams operation.

\section{Proof of the Verlinde formula for Higgs bundles}\label{Proof}

We recall that a principal $G$-Higgs bundle over $\Sigma$ is a pair $(E,\Phi)$, where $E$ is a holomorphic principal $G$-bundle and $\Phi\in H^0(\Sigma,\mathrm{ad}(E)\otimes K_\Sigma)$. 
For moduli space of such we have used the notation $M_H$, and for the stack, ${\mathfrak M}_H$.
   
We will also consider the moduli space (stack) of $R$-twisted Higgs bundles $M_H^{(R)}$ ($\frak{M}_H^{(R)}$), {\it e.g.}~pairs of a $G$-bundle $E$ and an $R$-twisted Higgs field in 
$$\Phi\in H^0(\Sigma,\mathrm{ad}(E)\otimes K^{R/2}_\Sigma),$$
 where $R$ is an integer. One gets the usual Higgs bundle for $R=2$ and the ``co-Higgs bundle'' for $R=-2$. See {\it e.g.}~\cite{SR} and references therein for more discussion about the latter.

Let $N^{(R)}_{\frak{M}}$ be the normal bundle of $\frak{M} \subset {\frak{M}}_H^{(R)}$. 
Then the stack $\frak{M}_H^{(R)}$ can be identified as
\be \label{SS}
\frak{M}_H^{(R)}=\mathrm{Spec}\,\mathrm{Sym}(N^{(R)*}_{\frak{M}}). 
\ee

By abuse of notation we also denote the determinant line bundle over both ${\mathfrak M}$ and ${\mathfrak M}^{(R)}_H$ simply by ${\mathfrak L}$. When we consider these bundles and their cohomology groups, it should be clear, from the notation usage and the context, which bundle is being referred to.

\begin{prop}
We have the following identification of cohomology groups
$$ H^i\left({\mathfrak M}^{(R)}_H, \frak{L}^k\right) = H^i\left({\mathfrak M}, S^*(N_{\mathfrak M}^{(R)*})\otimes \frak{L}^k\right)$$
for all $i$.
\end{prop}

\begin{proof}
We consider the spectral sequence for the fibration ${\mathfrak M}^{(R)}_H \ra {\mathfrak M}$ and use \eqref{SS} to provide a description of the fibers. From this, it is clear that the spectral sequence collapses at the $E_2$-page, and since the cohomology groups along the fibers concentrate in degree zero, we get the claimed result.
\end{proof}

The right-hand side of formula \eqref{CVF} is very reminiscent of the index formula of Teleman and Woodward. In particular, if one sets $R=0$, as one can explicitly check, the formula coincides with the index of $\lambda_{-t}(\Omega)\otimes L^k$, where 
\be
\Omega=R(\pi_{\frak{M}})_*(E(\frak{g})\otimes K_\Sigma)
\ee
is the cotangent complex of $\frak{M}$. 

In fact the following theorem for general $R$ can be proved simply by a straight forward computation involving Theorem \ref{TWInd}.

\begin{thm}\label{thm:Main5} The index of the K-theory class of $S_{t}(N^{(R)*}_{\frak{M}})\otimes \frak{L}^k$ on the stack $\frak{M}$ is given by
\be
\ind\left(\frak{M},S_{t}( N_{\mathfrak M}^{(R)*})\otimes \frak{L}^k\right) = \sum_{f\in F^{\rm reg}_{\rho,t} / W} \theta_{t, R}(f)^{1-g},
\ee 
with $\theta_{t, R}$ now given by
$$\theta_{t,R} = (1-t)^{(R-1)\mathrm{rk} G}\frac{1}{|F|} \frac{ \prod_{\alpha}(1-e^\alpha)(1-te^\alpha)^{R-1}}{\det H_t^\dagger}.
$$
\end{thm}

\begin{proof}
Define the K-theory class 
\be
\Omega_R=R(\pi_{\frak{M}})_*(E(\frak{g})\otimes K^{1-R/2}_\Sigma)
\ee
on ${\mathfrak M}$ and observe that
$$\Omega_R[1] = N_{\mathfrak M}^{(R)*}$$
at the level of K-theory. This follows from the exact sequence in Remark 2.7 in \cite{BR}.
Now we rewrite the class $\Omega_R$ in terms of the standard Atiyah-Bott generators using the family index theorem as follows. Let $T_{\pi_{\frak{M}}}$ be the relative tangent sheaf along the projection $\pi_{\frak{M}} : \frak{M} \times \Sigma \ra \frak{M}$ and let $x^*$ be the generator of $H^2(\Sigma)$, Poincare dual to a point $x\in \Sigma$, {\it e.g.}~$x^*\cap [\Sigma] = 1$. Then we have that
\begin{eqnarray*}\label{FamInd}
\mathrm{ch}(\Omega_R) & = & \mathrm{Td}(T_{\pi_{\frak{M}}})\cup\mathrm{ch}(E(\frak{g})\otimes K^{1-\frac{R}{2}}_\Sigma)\cap [\Sigma]
\\
&=& (1+(1-g)x^*)\cup \ch(E(\frak{g}))\cup (1-(1-g)(2-R)x^*)\cap [\Sigma]\\
&=&(1+(g-1)(R-1)x^*)\cup\ch(E(\frak{g}))\cap [\Sigma]\\
&=&\mathrm{ch}(E_{\Sigma}\frak{g})+(1-R)(g-1)\ch(E_x\frak{g}).
\end{eqnarray*}
From this we conclude that, at the level of rational K-theory, we have
\be\label{LambdaSheave}
\Omega_R=(E_{\Sigma}\frak{g})\oplus(1-R)(g-1)(E_x\frak{g}).
\ee

Now, it is easy to compute the index of $\lambda_{-t}(\Omega_R) \otimes L^k$ using Theorem \ref{TWInd}. Since the coefficients of $E_{\Sigma}\frak{g}$ in  \eqref{LambdaSheave}  has no $R$-dependence, the equation \eqref{beq} has no dependence on $R$ and the solution set just becomes $F_{\rho,t}$. The place where $R$ enters is through the factor
\be
\Tr_{\lambda_{-t}(\Lie{g})^{(1-R)(g-1)}} = (1-t)^{(1-R)(g-1) \rk G} \prod_{\alpha}(1-te^{\alpha})^{(1-R)(g-1)}
\ee
in the summand of the formula in Theorem \ref{thm:Main5}.
\end{proof}

\begin{thm}\label{Vanish}
For $R\geq 2$ and all positive integers $i$ and $k$ we have that
$$ H^i\left({\mathfrak M}, S^*(N_{\mathfrak M}^{(R)*})\otimes \frak{L}^k\right) = 0.$$
\end{thm}

We remark that the vanishing of the cohomology in negative degrees for $g>1$ also follows from the fact that, $\frak{M}$ is ``very good'' for $g>1$, as established in \cite{BD}. However, for $g=0$ we certainly know that there is non-trivial cohomology in negative degrees as we see from Proposition \ref{g=0}.

We present here a proof of Theorem \ref{Vanish}. The key idea of using the arguments presented in Section 4 of \cite{FT}, but now twisted by ${\mathcal O}(k)$ over the thick flag variety, is due to Constantin Teleman, and was past on to us by Nigel Hitchin. Hence, our arguments rely entirely on the beautiful, substantial and deep work by, first of all Teleman and then further his collaborations with in the first instance Fishel and Grojnowski \cite{FGT} and second with Frenkel \cite{FT}. Below we present our attempt to fill in the details of Teleman's suggestion. We strongly advice the reader to also consult Teleman's appendix to \cite{HL}, which for some parts actually proceeds at a somewhat different route then ours.

\begin{proof}[Proof of Theorem \ref{Vanish}] We present here the proof for the case where $R=2$, noting that the proof for $R>2$ is completely similar, one just have to twist with the appropriate power of $K_\Sigma$, which by the remarks in the second last paragraph on page 12 in \cite{FGT} does not cause any essential change in the arguments.

First we pick a point $x$ on $\Sigma$ and let $\Sigmao$ be the complement of $\{x\}$ on $\Sigma$. We shall further fix a formal coordinate $z$ centred at $x$, in which we do all Laurent expansions.
We use the presentation of the stack $\frak{M}$ as the double quotient $\frak{G}_0\backslash \bX$, where $\bX= \frak{G}/G[\Sigmao]$, $\frak{G} = G((z))$ and $\frak{G}_0 = G[[z]]$. The space $\bX$ is  the thick flag variety as in \cite{FT} and \cite{FGT}. Over $\bX$ we consider the bundle
$$ \Lie{g}[\Sigmao]_b = \left(\frak{G}\times \Lie{g}[\Sigmao]\right)/G[\Sigmao]$$
and the trivial bundle
$$ \tilde{\Lie{g}}_\tau = \bX\times \tilde{\Lie{g}},$$
where $\tilde{\Lie{g}} = \Lie{g}((z))/\Lie{g}[[z]]$.
Further we present the tangent complex of $\frak{M}$ as a two-step resolution
$$\partial: \Lie{g}[\Sigmao]_b \rightarrow \tilde{\Lie{g}}_\tau $$
where $\partial_{\varphi} = \text{Ad}_\varphi\circ {\mathcal E}_{x}$ at $\varphi\in \frak{G}_0$ and ${\mathcal E}_{x}$ is the Laurent expansion at $x$. 
For each non-negative integer $r$, we get the following complex 
\be\label{ess}
0 \ra \Lambda^r\Lie{g}[\Sigmao]_b \ra \Lambda^{r-1}\Lie{g}[\Sigmao]_b \otimes S^1( \tilde{\Lie{g}}_\tau) \ra \ldots  \ra S^r ( \tilde{\Lie{g}}_\tau)\ra 0.
\ee

In parallel to (4.3) of \cite{FT} we get that the $\frak{G}_0$-equivariant hyper-cohomology of this differential graded complex with $\Lie{g}[\Sigmao]_b $ in degree $-1$ gives
$$H^q(\frak{M},S^r T \otimes {\mathcal L}^k) = \bigoplus_{s+t=r} 
{\mathbb H}^q_{\frak{G}_0}\left( \bX, \Lambda^s\Lie{g}[\Sigmao]_b\otimes S^t( \tilde{\Lie{g}}_\tau)\otimes {\mathcal O}(k)\right).$$
As in \cite{FT}, we filter by $s$-degree to get a spectral sequence, whose $E_1$-page is
\begin{equation}
E_{1}^{-s,q} = H_{\frak{G}_0}^q\left(\bX, \Lambda^s\Lie{g}[\Sigmao]_b\otimes S^{r-s}( \tilde{\Lie{g}}_\tau) \otimes {\mathcal O}(k)\right).
\end{equation}

In order to compute this equivariant cohomology, we consider the Leray spectral sequence for the fibration ${\mathfrak M} \ra B\frak{G}_0$, whose   $E_2$-page is
$${\tilde E}_{2}^{p,q} = H^p(B\frak{G}_0, H^q\left(\bX, \Lambda^s\Lie{g}[\Sigmao]_b\otimes  {\mathcal O}(k)  \right) \otimes S^{r-s}( \tilde{\Lie{g}})).$$
By remark (8.10) in \cite{Te1}, we know that Theorem (0,5') also of \cite{Te1} applies to $H^q\left(\bX, \Lambda^s\Lie{g}[\Sigmao]_b\otimes  {\mathcal O}(k)  \right)$ since
$$ (\Lie{g}((z))/\Lie{g}[\Sigma^o])^* \cong \Lie{g}[\Sigma^o]$$
and we therefore get that
$${\tilde E}_{2}^{p,q} = \bigoplus_{\lambda \in \Lambda_k} \langle H^q\left(\bX, \Lambda^s\Lie{g}[\Sigmao]_b\otimes  {\mathcal O}(k)  \right)\mid {\mathbf H}_\lambda^{(k)} \rangle \otimes H^p_{\frak{G}_0}({\mathbf H}_\lambda^{(k)}\otimes S^{r-s}( \tilde{\Lie{g}})),$$
where the multiplicity space is given by ($V_\lambda$ is dual of the finite dimensional irreducible representation of $G$ associated to $\lambda$)
\be\label{MRLCX}
\langle H^q\left(\bX, \Lambda^s\Lie{g}[\Sigmao]_b\otimes  {\mathcal O}(k)  \right)\mid {\mathbf H}_\lambda^{(k)} \rangle  = H^q ({\mathfrak M}, \Lambda_{\bX}^s \otimes V_\lambda \otimes L^k).
\ee
Here $\Lambda_{\bX}^s$ denotes the descended bundle from $\bX$ of $\Lambda^s\Lie{g}[\Sigmao]_b$. We now use Theorem 3 of \cite{Te1/2} combined with Remark (8.10) from the same reference to conclude that
\be\label{MRLC}
H^q ({\mathfrak M}, \Lambda_{\bX}^s \otimes V_\lambda \otimes L^k) = H^q_{G[\Sigma^0]}( {\mathbf H}_0^{(k)} \otimes \Lambda^s\Lie{g}[\Sigmao] \otimes V_\lambda).
\ee

Now observe that there is a perfect pairing
$$ (\Lie{g}((z))/ \Lie{g}[[z]]) \times \Lie{g}[[z]] \ra {\mathbb C},$$
using a combination of the invariant inner product in $\Lie{g}$, the product of Laurent series and the residue of the resulting complex valued Laurent series. But then we conclude that
$$ H_{\frak{G}_0}^{*}\left(S(\tilde{\Lie{g}}) \otimes {\mathbf H}_\lambda^{(k)} \right) = 
 H_{\frak{G}_0}^{*}\left(S(\Lie{g}[[z]]_{\rm res}^*) \otimes {\mathbf H}_\lambda^{(k)} \right).$$
 Using the argument in (4.11) of \cite{FGT}, we get that
 $$ H^*_{\rm res} (\Lie{g}[z], \Lie{g}; S\Lie{g}[[z]]^*_{\rm res}\otimes {\mathbf H}_\lambda^{(k)}) = 
 H_{\frak{G}_0}^{*}\left(S(\Lie{g}[[z]]_{\rm res}^*) \otimes {\mathbf H}_\lambda^{(k)} \right).$$
But by Theorem E of \cite{FGT} we get that
\begin{equation}\label{Valg}
H^p_{\rm res} (\Lie{g}[z], \Lie{g}; S\Lie{g}[[z]]^*_{\rm res}\otimes {\mathbf H}_\lambda^{(k)}) = H^{p}\left(\Lie{g}[[z]], \Lie{g}; S(\Lie{g}[[z]])^*) \otimes {\mathbf H}_\lambda^{(k)} \right) =0
\end{equation}
for $p>0$.  
We conclude thus that
\be\label{E1}
\phantom{hhh} E_{1}^{-s,q} = \bigoplus_{\lambda \in \Lambda_k}   H^q_{G[\Sigma^0]}( {\mathbf H}_0^{(k)} \otimes \Lambda^s\Lie{g}[\Sigmao] \otimes V_\lambda) \otimes \left({\mathbf H}_\lambda^{(k)} \otimes S^{r-s}( \tilde{\Lie{g}})\right)^{\frak{G}_0}.
\ee

We see now by the above and Proposition \ref{prop1} below that $E_{1}^{-s,q}$ is concentrated in the region below the diagonal with $q\leq s$, ensuring the vanishing of all positive degree cohomologies in the abutment.

\end{proof}

We propose the following generalization of Teleman's group cohomology results in \cite{Te1}, which we needed in the above argument.

\begin{prop}\label{prop1}
We have the following vanishing of the group cohomology
\be\label{RV}
H^q_{G[\Sigmao]}( {\mathbf H}_0^{(k)} \otimes \Lambda^s\Lie{g}[\Sigmao] \otimes V_\lambda) =0
\ee
for all $q> s$.
\end{prop}

We will follow the lines of the arguments presented in \cite{Te0} and \cite{Te1/2} in order to establish the vanishing of this cohomology by a combination of relations between the group cohomology and  Lie-algebra cohomology, invariance of the Lie-algebra cohomology under degeneration of the curve $\Sigmao$ to a nodal curve, followed by the behavior of the Lie algebra cohomology under resolution of the nodal curve and finally a control of the cohomology for once-punctured ${\mathbb P}^1$. However, these steps are best performed in reverse order. 

Let $\Sigma'$ be a smooth curve obtained from a compact curve $\overline{\Sigma}'$ by removing one point.
We start by recalling that Proposition 10.5 from \cite{FGT} applies to $\Sigma'$, but as remarked in \cite{FT} on page 12, one can twist with a line bundle, which we in this case choose to be $T\Sigma'$ to obtain that

\begin{thm}[Fishel, Grojnowski, Teleman]
We have a natural isomorphism
$$ H^*(\Lie{g}[\Sigma'], \Lambda^s(\Lie{g}[\Sigma'])) = H^*(\bX_{\Sigma'},  \Lambda^s\Lie{g}[\Sigma']_b) \otimes H^*(G(\Sigma),{\mathbb C}).$$
\end{thm}
 Further Theorem D,(i) form the same reference applied to $\Sigma'$ gives that
 \begin{thm}[Fishel, Grojnowski, Teleman]
 The ring $H^*(\bX_{\Sigma'}, \Lambda^*\Lie{g}[\Sigma']_b)$ is the free skew-commutative algebra generated by copies of $\Omega^0[T\Sigma']$ and $\Omega^0[\Sigma']$ in $$H^m(\bX_{\Sigma'}, \Lambda^m\Lie{g}[\Sigma']_b) \text{ and }H^m(\bX_{\Sigma'}, \Lambda^{m+1}\Lie{g}[\Sigma']_b)$$ respectively.
 \end{thm}

We now propose the following generalization of Theorem 2.4 in \cite{Te1/2}.

\begin{thm}\label{2.4C}
We have that
$$ H^*(\Lie{g}[\Sigma'], {\mathbf H}_0^{(k)} \otimes \Lambda^s\Lie{g}[\Sigma'] \otimes V_\lambda) \cong H^0(\Lie{g}[\Sigma'], {\mathbf H}_0^{(k)}  \otimes V_\lambda) \otimes H^*(\Lie{g}[\Sigma'], \Lambda^s(\Lie{g}[\Sigma'])).$$
\end{thm}

We will now assume Theorem \ref{2.4C} and following \cite{Te1/2}, consider the van Est spectral sequence which relates the cohomology of $G[\Sigma']$ and that of the Lie algebra $\Lie{g}[\Sigma']$. We recall Propostion 6.1 from this reference

\begin{prop}[Teleman]\label{specTel}
There is a spectral sequence with
$$ E^{p,q}_2 = H^p_{G[\Sigma']}({\mathbf H}_0^{(k)} \otimes \Lambda^s\Lie{g}[\Sigma'] \otimes V_\lambda \otimes H^q_s(G[\Sigma'])),$$
which converges to $H^*(\Lie{g}[\Sigma'], {\mathbf H}_0^{(k)} \otimes \Lambda^s\Lie{g}[\Sigma'] \otimes V_\lambda) $ and is compatible with products.
\end{prop}

It is now simple to argue as in section VII.1  and VII.2 in \cite{Te1} leading up to (7.3), that the edge homomorphism 
$$ H^{p+q}(\Lie{g}[\Sigma'], {\mathbf H}_0^{(k)} \otimes \Lambda^s\Lie{g}[\Sigma'] \otimes V_\lambda) \ra H^p_{G[\Sigma']}({\mathbf H}_0^{(k)} \otimes \Lambda^s\Lie{g}[\Sigma'] \otimes V_\lambda )\otimes H^q_s(G[\Sigma'] )$$
in the van Est spectral sequence from Proposition \ref{specTel} will be surjective for $p>s$, hence we will get the vanishing of the group cohomology in Proposition \ref{prop1}.

\begin{proof}[Proof of Theorem \ref{2.4C}]
This proof proceeds as the proof of theorem 2.4.\ in \cite{Te1}.
We can place $\bar\Sigma'$ in a family of curves over some parameter space, with a section through $\bar \Sigma'-\Sigma'$ over some point in the parameter space, in such a way that the fiber over some other point, say $\Sigma_0$, is a nodal degeneration of $\bar\Sigma'$, whose normalization $\tilde{\Sigma}_0$ is a copy of ${\mathbb P}^1$. Let $\Sigmao_0$ denote the complement of the value of the section in $\Sigma_0$. We assume this point is not one of the $g$ nodes of $\Sigma_0$. We can now apply Theorem 2.3 in \cite{Te1/2} to 
obtain
$$ H^q(\Lie{g}[\Sigmao], \Lie{g}; {\mathbf H}_0^{(k)} \otimes \Lambda^s\Lie{g}[\Sigmao_0] \otimes V_\lambda) \cong H^q(\Lie{g}[\Sigmao_0], \Lie{g}; {\mathbf H}_0^{(k)} \otimes \Lambda^s\Lie{g}[\Sigmao_0] \otimes V_\lambda) .$$

Theorem 2.3 in \cite{Te1/2} applies directly in the case where $s=0$ and the proof applies to the case $s>0$ as well, as we will now argue. We consider the following filtration on $\Lie{g}(\Sigmao_t)$
$$ {\mathcal F}^d(\Lie{g}[\Sigmao_t]) = \{ \xi\in \Lie{g}[\Sigmao_t]\mid  \text{ord}_{p_i}(\xi) \leq d\},$$
where $t$ runs through the parameter space. We observe that the action of $\Lie{g}[\Sigmao_t]$ on ${\mathbf H}_0^{(k)} \otimes \Lambda^s\Lie{g}[\Sigmao_0] \otimes V_\lambda$ is filtered, where we use the filtration on ${\mathbf H}_0^{(k)}$ induced by the energy grading used in \cite{TUY}. Furthermore we see that the associated graded space to this filtrated space ${\mathcal F}^d(\Lie{g}[\Sigmao_t])$, is finite-dimensional in each piece. The discussion in Section 4 and 5 of \cite{Te1/2} can now be applied as is for the $\Lie{g}[\Sigmao_t]$-module  $\rm{Gr}({\mathbf H}_0^{(k)} \otimes \Lambda^s\Lie{g}[\Sigmao_0]) \otimes V_\lambda$, to establish that the corresponding homology sheaves of our family of pointed curves of the affine line are first coherent and secondly locally free, which gives the desired result.

Continuing with the argument in \cite{Te1/2}, now entering the part contained in Section 3 of that reference, we define the subalgebra $\frak{h} \subset \Lie{g}[\Sigmao_0]$ to be the kernel of the evaluation map at the nodes. We then have the following commutative diagram of short exact sequences of Lie-algebras
\be \label{CDLie}
\phantom{hhhhh}\begin{CD}
 @.  0   @. 0 @.  0 @.\\
@. @VVV        @VVV  @VVV @.\\
0 @>>>  \frak{h}   @>>>  \frak{h} @>>>  0 @>>> 0\\
@. @VVV        @VVV  @VVV @.\\
0 @>>>  \Lie{g}[\Sigmao_0]  @>>>  \Lie{g}[\tilde{\Sigma}^\circ_0]@>>>  \Lie{g}^{g} @>>>  0\\
@. @VVV        @VVV  @VVV @.\\
0 @>>>  \Lie{g}^{g} @>>>  \Lie{g}^{2g} @>>>  \Lie{g}^{g} @>>>  0\\
@. @VVV        @VVV  @VVV @.\\
 @.  0   @. 0 @.  0 @. 
\end{CD}
\ee

We now consider the spectral sequences $E$ and $\tilde{E}$ for the second and third columns respectively, with coefficients in the relevant modules, together with the induced map between them from the above diagram. Now the argument proceeds exactly as on page 256 and 257 in \cite{Te1/2}, again using the obvious filtration ${\mathcal F}$ from above, to get semi-simplicity in the same way as is done at this place in \cite{Te1/2}. 

For the once-punctured ${\mathbb P}^1$, we need a replacement of Theorem 3.1 in \cite{Te1/2}.
To this end, we can simply apply the proof of Theorem F in \cite{FGT}. Let $X = G((z))/G[z^{-1}]$. Without the assumption of $G$ being simply-laced, we can still consider the same spectral sequence, which converges to $H^*(\frak{M}({\mathbb P}^1), \Omega^*(k)) = H^{*,*}(BG, {\mathbb C})$ and whose $E_1$-page is given by 
$$ E^{p,q}_1 = \bigoplus_{\lambda\in \Lambda_k} H^q(X, \Omega_X^{r-p}(k)\otimes V_\lambda)  \otimes \left({\mathbf H}_\lambda^{(k)} \otimes S^{p}( \Lie{g}[[z]]^*)\right)^{\frak{G}_0}.$$
But we know $H^{*,*}(BG, {\mathbb C})$ thus we get Theorem \ref{2.4C} in this simple ${\mathbb P}^1$ case with one marked point, which completes the proof.
\end{proof}

Another approach to a proof of Theorem \ref{Vanish} is to use the pushforward of ${\mathcal L}^k$ from ${\frak M}$ along the Hitchin map to the Hitchin base. Over the open of the base, where the Hitchin map is a fibration, positivity of  ${\mathcal L}^k$ gives vanishing of all non-zero direct images. However, an argument is need to extend this to all fibers of the Hitchin map. Serre duality might actually do this along all fibers, however there could be problems due to very singular and/or non-reduced fibers. We hope to be able to complete this alternative argument at some point.

We expect that $ H_{G[[z]]}^{*}\left(S^{p}(\Lie{g}((z))/\Lie{g}[[z]]) \otimes {\mathbf H}_\lambda^{(k)} \right) $ vanished in positive degree can also be obtained by an explicit Laplace type argument \`a la Theorem E from \cite{FGT} directly based on a positivity result for the Laplacian in higher cohomological degrees. 

The arguments given in the upper part of page 13 in \cite{FT}, which again refer back to \cite{FGT} Theorem A, should provide us with a description of 
$$ H_{G[[z]]}^{0}\left(S^{p}(\Lie{g}((z))/\Lie{g}[[z]]) \otimes {\mathbf H}_\lambda^{(k)} \right) $$ 
generated over $ H_{G[[z]]}^{0}\left( {\mathbf H}_\lambda^{(k)} \right) = H^0(\frak{M},\CL^k\otimes V_\lambda)$ by generators in $p=m$, where $m$ runs through the exponents of $\Lie{g}$, of copies of $({\mathbb C}((z))/{\mathbb C}[[z]]) \otimes (\partial/\partial z)^{\otimes m}$. We expect that combining such generators with arguments along the line of the arguments presented in the proof of Theorem 4.2. of \cite{FT}, regarding the leading differentials action on the generators should allow for a complete description of the space of sections $H^0(\frak{M}, S^*T \otimes L^k)$.

In any case, by this vanishing of higher cohomology groups, we now get the Complex Verlinde formula in its $R$-twisted generalization for the stack.

\begin{thm}\label{thm12}
For $R,g\geq 2$ and we have the following graded dimension formula
\be
\dim_t H^0\left(\frak{M},S_{t}( N_{\mathfrak M}^{(R)*})\otimes \frak{L}^k\right) = \sum_{f\in F^{\rm reg}_{\rho,t} / W} \theta_{t, R}(f)^{1-g}.
\ee 
\end{thm}

Since for $R,g\geq 2$, the stack $(\frak{M}_H^{(R)})^{\text{ss}}$ of semi-stable $G$-Higgs bundles has codimension more than two in $\frak{M}^{(R)}_H$, 
we observe that
$$ H^0(\frak{M}^{(R)}_H, \frak{L}^k) = H^0((\frak{M}^{(R)}_H)^{\text{ss}}, \frak{L}^k) = H^0(M^{(R)}_H, L^k).$$
and we get the following Verlinde formula for Higgs bundles in the R-twisted case.
\begin{thm}\label{thm13}
For $R,g\geq 2$ we have the following graded dimension formula
\be
\dim_t H^0\left(M^{(R)}_H,L^k\right) = \sum_{f\in F^{\rm reg}_{\rho,t} / W} \theta_{t, R}(f)^{1-g}.
\ee 
\end{thm}

The complex Verlinde formula, {\it e.g.}~Theorem \ref{thm:Main1}, stated in the introduction, now follows by taking $R=2$ in Theorem \ref{thm13}.

For the case of $R\leq 0$, we will in general not have the vanishing of non-zero cohomology groups, neither would we be able to directly relate quantization of the stack to that of the space, as $\frak{M}_H$ is a derived stack for $g>1$. However, we can offer the following alternative interpretation of our index formula in this case, simply by noting that
$$N_{\frak{M}}^{(R)*} = - N_{\frak{M}}^{(2-R)},$$
and that the non-zero degree cohomology of $\lambda_{-t}( N_{\mathfrak M}^{(R)})\otimes \frak{L}^k$ vanishes, thus we conclude
\begin{thm}
For $R\leq 0$ and $g\geq 2$ we get that
\be
\dim H^0\left(\frak{M},\lambda_{-t}( N_{\mathfrak M}^{(R)})\otimes \frak{L}^k\right) = \sum_{f\in F^{\rm reg}_{\rho,t} / W} \theta_{t, 2-R}(f)^{1-g}.
\ee 
\end{thm}

For $R=0$ we recover Theorem 6.4 in \cite{TW}, where $U$ is set to be the trivial representation.

\section{Proof of the Verlinde formula for parabolic Higgs bundles}\label{ParaProof}
We recall that a parabolic Higgs bundle is a pair $(B,\Phi)$ where $B$ is a parabolic bundle with reduction of the structure group of $B$ to a parabolic subgroup $P_i$ over each parabolic point $\{x_1,\ldots,x_n\}$  on $\Sigma$. Fix the divisor $D = \sum_i x_i$. We assign integral dominant weights $\vec \lambda = (\lambda_1, \ldots, \lambda_n)$, one for each point, such that $P_i$ preserves $\lambda_i$
. As for the Higgs field, we require that 
$$\Phi \in H^0(\Sigma, \text{ad}_s(B)\otimes K^{R/2}(D)),$$
in the $R$-twisted case, which we will consider. Further, the notation $\text{ad}_s(B)$ refers to the adjoint bundle of $B$ with the reduction of the structure group corresponding to $P_i$ over $x_i$.  We let $P$ be the corresponding moduli space of semi-stable parabolic $G$-bundles and $P_H$ the moduli space of such $G$-Higgs bundles. The corresponding stacks we denote ${\mathfrak P}$ and ${\mathfrak P}^{(R)}_H$ respectively. By slight abuse of notation, we also use $B$ to denote the universal bundle over $\Sigma\times \frak{P}$, and $B(\frak{g})$ the associated adjoint bundle. Let $N^{(R)}_\frak{P}$ denote the normal bundle to $\frak{P}$ in $\frak{P}^{(R)}_H$.

As discussed in the introduction, each weight is required to be of level $k$ and we then get a level-$k$ determinant line bundle $L_k$ over the parabolic moduli space $P$ and ${\mathfrak L}_k$ over the stack  ${\mathfrak P}^{(R)}_H$.

Since 
\be\label{ParaSym}
\frak{P}^{(R)}_H=\mathrm{Spec}\,\mathrm{Sym}(N^{(R)*}_{\frak{P}}), 
\ee
we can use the same techniques as in the non-parabolic case to prove the following proposition.
\begin{prop}
We have the following identification of cohomology groups
$$ H^i({\mathfrak P}^{(R)}_H, \frak{L}_k) = H^i({\mathfrak P}, S^*(N^{(R)*}_{\mathfrak P})\otimes \frak{L}_k)$$
for all $i$.

\end{prop}

Using this in cohomological degree zero and ${\mathbb C}^*$-degree $n$, we in particular get that
$$ H_n^0(\frak{P}_H, \frak{L}_k) = H^0(\frak{P}, S^n(N^{(R)*}_{\frak{P}})\otimes \frak{L}_k).$$

 The stack ${\mathfrak P}$ can be viewed as the total space of a fibration over ${\mathfrak M}$
\be
G/P_1\times\ldots\times G/P_n\longrightarrow \frak{P} \stackrel{\pi}{\longrightarrow} \frak{M}.
\ee
The pullback of $L_k$ to $G/P_i$ becomes $L_{\lambda_i}$, whose holomorphic sections give the irreducible representation $V_{\lambda_i}$ associated to $\lambda_i$. 

We consider the diagram of inclusions of stacks for each point $p\in {\mathfrak M}$
$$
\begin{CD}
(\pi^{(R)} )^{-1}(p)    @>>>  {\mathfrak P}_H^{(R)} @>\pi^{(R)}>>  {\mathfrak M}^{(R)}_H\\
@VVV        @VVV  @VVV \\
 \pi^{-1}(p)  @>>>  {\mathfrak P} @>\pi>>  {\mathfrak M}.
\end{CD}
$$
Around the last square in this diagram,  we consider the the exact sequence of the tangent complexes and normals
\be \label{CDN}
\phantom{hhhhh}\begin{CD}
0 @>>>  T_{\pi^{(R)}}{\mathfrak P}_H^{(R)}    @>>>  T{\mathfrak P}_H^{(R)} @>d\pi_{H}>>  \pi^*T{\mathfrak M}^{(R)}_H @>>> 0\\
@. @VVV        @VVV  @VVV @.\\
0 @>>>  T_\pi {\mathfrak P}  @>>>  T{\mathfrak P} @>d\pi>>  \pi^*T{\mathfrak M} @>>>  0\\
@. @VVV        @VVV  @VVV @.\\
0 @>>>  T_{\pi^{(R)}}{\mathfrak P}_H /T_\pi {\mathfrak P} @>>>  N^{(R)}_{\mathfrak P} @>d\pi>>  \pi^*N^{(R)}_{\mathfrak M} @>>>  0.
\end{CD}
\ee

We have the following fibered product presentation
\be\label{prodP}
\frak{P} = E_{1}/P_i\times_\frak{M}  \ldots \times_\frak{M} E_n/P_n,
\ee
where $E_i$ is the pull back of the universal bundle $E$ under the inclusion map 
$$ \{x_i\}\times \frak{M} \subset \Sigma \times \frak{M}.$$
From this we immediately conclude that
\begin{lemma}\label{lemtan}
We have the following isomorphism of bundles over ${\mathfrak P}$
\be
T_{\pi^{(R)}}{\mathfrak P}_H /T_\pi {\mathfrak P} = \bigoplus_{i=1}^n p_i^*(E_{x_i} \times \Lie{g})/P_i
\ee
where $p_i$ is projection on the $i$'th factor in \eqref{prodP}.
\end{lemma}

Similar to the non-parabolic case,  we get an identification
$${N_{\frak{P}}}^* = \Omega_{R,D}[1]$$
where
$$\Omega_{R,D}=R\pi_*\left(B(\frak{g})\otimes K^{1-R/2}(-D)\right).$$
This follows again from infinitesimal deformation considerations presented in \cite{BR} for parabolic bundles in Section 6 of that paper.

For $R\leq0$ the stack $\frak{P}$ is entirely derived, but similar to what we did in the non-parabolic case, we can also consider  the following K-theory class 
$$\Omega'_{R,D}=R\pi_*\left(B(\frak{g})\otimes K^{1-R/2}(D)\right). $$

We now compute the index of $S_t(N_\frak{P}^*)\otimes \frak{L}_k$ over $\frak{P}$ by first pushing forward along $\pi$ to $\frak{M}$ followed by an application of the Teleman-Woodward index formula over $\frak{M}$. For the classes that we are interested in, we see (details given in the proof of Theorem \ref{thm:MainP} below) from Lemma \ref{lemtan}, that we in particular need to evaluate the index of the classes $L_{\lambda_i}\otimes S_t (T(G/P_i))$ and $L_{\lambda_i}\otimes \lambda_{-t} (T^*(G/P_i))$ over the flag variety $G/P_i$. These indices, which we will call the ``deformed characters'', can be computed using a fix-point formula as we will now argue.

\begin{definition}
For a weight $\lambda$ left invariant under a parabolic subgroup $P'$,
we define the deformed characters for the pair $(\lambda, P')$\footnote{To avoid clutter, $t$ in the subscript is often omitted. $P'$ will be suppressed when it is the standard Borel subgroup. }
\be\label{DCF}
\Theta_{\lambda,P',t} = \sum_{w\in W}\frac{e^{w(\lambda)}}{\prod_{\alpha>0}\left(1-e^{-w(\alpha)}\right)\prod_{\alpha\in \frak{g}/\frak{p'}}\left(1-te^{w(\alpha)}\right)}
\ee
and
\be\label{DefChar'}
\Theta'_{\lambda,P',t}=\sum_{w\in W}\frac{e^{w(\lambda)}\prod_{\alpha\in \frak{g}/\frak{p'}}\left(1-te^{-w(\alpha)}\right)}{\prod_{\alpha>0}\left(1-e^{-w(\alpha)}\right)}.
\ee
\end{definition}

For $t=0$, $\Theta_{\lambda,P'}$ and $\Theta'_{\lambda,P'}$ have no dependence on $G'$ and indeed reproduce the character of the representation $R_\lambda$ associated with $\lambda$. When $P'$ is the Borel subgroup of $G$, the $\Theta'_{\lambda}$'s are, modulo a normalization factor, the Hall-Littlewood polynomials.

\begin{lemma}\label{IndT}
The deform characters gives the following $T$-equivariant indices 
\be
\label{eqind+}
\mathrm{Ind}_{T}(G/P',L_\lambda^k\otimes S_t T(G/P'))= \Theta_{\lambda,P'}
\ee
and
\be
\label{eqind-}
\mathrm{Ind}_{T}(G/P',L_\lambda^k\otimes \lambda_{-t} T^*(G/P'))= \Theta'_{\lambda,P'}.
\ee
\end{lemma}

\begin{proof}
We first consider the case with $P'=B$, for which the lemma was already proved in \cite{G}. In this case, the left action of $T$ on $G/B$ has fixed points labeled by $w\in W$, since by the very definition, the Weyl group is $W= N(T)/T$. Now  the action of a $\xi\in T$ on $T_w(G/B)$ is specified by $e^{w(\alpha)}(\xi)$, where $\alpha$ runs through the positive roots $\frak{R}_+$, and hence we get that
$$ \tr\left(\xi : (L_\lambda)_w\otimes S_t(T_w(G/B)) \ra (L_\lambda)_w \otimes S_t(T_w(G/B))\right) = \frac{e^{w(\lambda)}}{\prod_{\alpha\in \frak{R}_+}\left(1-te^{w(\alpha)}\right)}(\xi).$$
and further by dividing by the determinant of one minus the cotangential action at $w$, and summing over $W$, we get exactly formula \eqref{eqind+}.
For the case where we consider $\lambda_{-t}T^*(G/B)$, we get from the trace
\begin{eqnarray} 
\lefteqn{\tr\left(\xi : (L_\lambda)_w\otimes \lambda_{-t}(T_w^*(G/B))  \ra (L_\lambda)_w \otimes \lambda_{-t}(T_w^*(G/B))\right) } \phantom{xxxxxxxxxxxxxxxxxxxxx}\nonumber\\
&=& e^{w(\lambda)}\prod_{\alpha\in \frak{R}_+}\left(1-te^{-w(\alpha)}\right)(\xi).\nonumber
\end{eqnarray}
Further, for the case of $B \subset P' \subset G$, we can consider the projection from $G/B$ to $G/P'$ to get the desired result. Namely, $L_\lambda\otimes S^*T_{G/P'}$ is the pushforward of the sheave $L_\lambda\otimes S^*\left((\frak{g}/\frak{p}')\times_G G/B\right)$, whose $T$-equivariant index can be computed in exactly the same way but with the replacement of $\prod_{\alpha>0}$ by $\prod_{\alpha\in \frak{g}/\frak{p}'}$. Applying a similar trick for the sheave $L_\lambda\otimes \lambda_{-t}T_{G/P'}^*$ completes the proof of the lemma.

\end{proof}

\begin{thm}\label{thm:MainP} The index over $\frak{P}$ of $\CL^k\otimes \mathrm{S}_t (N^{(R)*}_{\frak{P}})$ is given by
$$ 
\mathrm{Ind}\left(\frak{P},\CL^k\otimes \mathrm{S}_t (N^{(R)*}_{\frak{P}})\right)=\sum_{f\in F^{\rm reg}_{\rho,t} / W} \theta_t(f)^{1-g} \prod_i\Theta_{\lambda_i,P_i}(f),
$$
where $\Theta_{\lambda}$ is given by \eqref{DCF}, which also equals
$$
\Theta_{\lambda,P'} = \sum_{w\in W} \frac{(-1)^{l(w)} e^{w(\lambda+\rho)}}{\Delta \prod_{\alpha\in \frak{g}/\frak{p}'}\left(1-te^{w(\alpha)}\right)}.
$$
\end{thm}

\begin{proof}
We consider the fibration
$$\pi : \frak{P} \ra \frak{M},$$
and consider the presentation \eqref{prodP}. By \eqref{CDN} we get the following equation in K-theory on $\frak{P}$
$$ N^{(R)}_{\mathfrak P}  = T_{\pi^{(R)}}{\mathfrak P}_H /T_\pi {\mathfrak P} +  \pi^*N^{(R)}_{\mathfrak M}. $$
So
$$ S_t(N^{(R)}_{\mathfrak P}) = S_t(T_{\pi^{(R)}}{\mathfrak P}_H /T_\pi {\mathfrak P})\otimes S_t\left( \pi^*N^{(R)}_{\mathfrak M}\right)$$
and we further get the following expression for $L_k$
$$L_k = \pi^*(L^k) \otimes \bigotimes_{i} p_i^*L^P_{\lambda_i},$$
where $L^P_{\lambda_i}$ is the line bundle over $E_i/P_i$ associated to the character $\lambda_i$ of the bundle $E_i$.
By the push-pull formula we get that
$$\pi_*(S_t(N^{(R)}_{\mathfrak P}) \otimes L_k) = \pi_*\left( S_t(T_{\pi^{(R)}}{\mathfrak P}_H /T_\pi {\mathfrak P}) \otimes  \bigotimes_{i} p_i^*L^P_{\lambda_i}\right) \otimes S_t\left(N^{(R)}_{\mathfrak M}\right)\otimes L^k.$$
By combining Lemma \ref{IndT} with Lemma \ref{lemtan} we conclude the following formula in  K-theory of $\frak{M}$
$$  \pi_*\left( S_t(T_{\pi^{(R)}}{\mathfrak P}_H /T_\pi {\mathfrak P}) \otimes  \bigotimes_{i} p_i^*L^P_{\lambda_i}\right) = \bigoplus_{i=1}^n E_{x_i} \Theta_{\lambda_i,P_i}.$$
But now the theorem follows from using the expression for $S_t(N^{(R)}_{\mathfrak M})\otimes L_k$ in K-theory of $\frak{M}$ from Section \ref{Proof} and the Teleman-Woodward index theorem.
\end{proof}

\begin{thm}
For $R\geq 2$ and all positive integers $i$ and $k$ we have that
$$ H^i({\mathfrak P}, S^*(N^{(R)*}_{\frak{P}})\otimes \frak{L}_k) = 0.$$
\end{thm}

The proof of this theorem is completely analogous to the proof of Theorem \ref{Vanish}, by the arguments presented in the first paragraph of section VII.4 in \cite{Te1}. Combined with vanishing theorems over $\frak{M}^{(R)}_H$, we have the {\em Verlinde formula for parabolic Higgs bundles}
\begin{thm}\label{thm:ParaVerlinde}
For $R,g \ge 2$ we have that 
$$
\dim_t H^0(P_H^{(R)}, L_k)=\mathrm{Ind}(\frak{P},L_k\otimes \mathrm{S}_t N^{(R)*}_{\frak{P}})=\sum_{f\in F^{\rm reg}_{\rho,t} / W} \theta_t(f)^{1-g} \prod_i\Theta_{\lambda_i,P_i}(f). 
$$
\end{thm}

Further, we also get an analog of Theorem \ref{thm:Main3} in the parabolic case

\begin{thm}\label{thm:MainP3}
Provided $g>2$ and $R\ge 2$, we have for all non-negative integers $n$ a natural isomorphism
$$ H^0(\frak{P}, S^{n}N^{(R)}_{P} \otimes {\mathcal L}_R^k) \cong H^0(P, S^{n}N^{(R)}_{P} \otimes {\mathcal L}_R^k).$$
\end{thm}
Again, for $R\le 0$, one can relate the index with another index over the moduli stack of $(2-R)$-twisted Higgs bundles. For that, one considers instead
$$\Omega'_{2-R,D} = R\pi_*(B(\frak{g})\otimes K^{R/2}(D))[1].$$
Just as for the non-parabolic case, we have  
\begin{thm}\label{thm:NegativeRP} For $R\leq 0$, we have that
$$\dim H^0({P}_H, {L}_k\otimes \lambda_{-t}N^{(2-R)}_{M})= \sum_{f\in F^{\rm reg}_{\rho,t} / W} \theta_{t,R}(f)^{1-g}\sum_i\Theta'_{\lambda_i,P_i}.$$
\end{thm}

\section{Low-genus exceptional cases}\label{Exceptional}

When the Riemann surface $\Sigma$ has small genus, the moduli stack of (parabolic) Higgs bundles $\mathrm{Spec}\,\mathrm{Sym}\left(N^{(R)*}_{\frak{M}}\right)$ with $R\ge 2$ can be an honest derived stack with a non-trivial derived structure. As a consequence, not only is there a drastic difference between the moduli space and the moduli stack, but also $H^*_n\left(\frak{M}_{H}^{(R)},\frak{L}^k\right)$ can have negative degree cohomology groups. One may attempt to stay away and regard them as pathological. But we will do exactly the opposite, because simple Riemann surfaces serve as the building blocks of the 2D TQFT that we will discuss in the next section, making the indices associated with them among the most interesting ones. In this section, we will look at several special cases with small genera. For simplicity, we will henceforth set $R=2$.

When $g=0$, the moduli stack $\frak{M}$ of $G$-bundles contains the classifying stack $BG$ as the substack ${\frak{M}^{\text{ss}}}$ of semi-stable bundles. Over ${\frak{M}^{\text{ss}}}$, the tangent complex is $\frak{g}[1]$, and $\frak{L}^k$ becomes trivial. Then, 
$$\mathrm{S}^r (T_{\frak{M}^{\text{ss}}})=\Lambda^r \frak{g}[r]$$ 
is concentrated in degree $-r$. And for any fixed $r$, the only non-vanishing cohomology group is $H^{-r}(BG,\Lambda^\bullet \frak{g}[r])=H^0(BG, \Lambda^r \frak{g})$. $H_G^0(\Lambda^\bullet \frak{g})$ is a skew-symmetric algebra with generators in degree $2m_i+1$, where $m_i$ are the exponents of $G$. Theorem~\ref{HigherVanish}, which we will prove by the end of this section, ensures that the contributions from unstable strata vanish when $k>2h-2$. Then we have:
\begin{prop} \label{g=0}
For $g=0$ and $k>2h-2$,
$$
\mathrm{Ind}(\frak{M},S_tT_{\frak{M}}\otimes \frak{L}^k)=\prod_{i=1}^{\mathrm{rk} G}(1-t^{2m_i+1})=P_{-t}(G),
$$
where $P_{-t}(G)$ is the Poincar\'e polynomial of the group $G$ in the variable $-t$. 
\end{prop}
This agrees with results of \cite{GP} and \cite{GPYY}, where it was found that the ``equivariant Verlinde formula'' on $S^2$ takes the curious values
$$
1-t^3 \quad \text{and} \quad  1-t^3-t^5+t^8
$$
for $K=SU(2)$ and $K=SU(3)$ with $k\geq 3$ and $k\geq 5$ respectively. Notice that negative coefficients appearing here always come from non-vanishing negative degree cohomology groups in $H^*(\frak{M}_H,\frak{L})$. Accounting for their contribution is crucial to ensure the functoriality of the 2d TQFT.

As the general index formula still applies for $g=0$, we obtain the following identity concerning $\theta_t$ evaluated on $F_{\rho,t}$.
\begin{prop} For $k>2h-2$, the structure constants $\theta_t(f)$ satisfies the following equation,
\be
\sum_{f\in F^{\mathrm{reg}}_{\rho,t}/W}\theta_t(f)=P_{-t}(G).
\ee
\end{prop}

When ${\mathbb P}^1$ has two marked points with parabolic weights $\lambda_1$ and $\lambda_2$ and the structure group is reduced to $B\subset G$ over both these two points, the index associated with this geometry is another fundamental building block for the 2D TQFT. We claim the following.

\begin{thm} \label{tqftp}
We have that
\be\label{d02}
d_{0,2}(\lambda_1,\lambda_2)= \delta_{\lambda_1,\lambda_2^*}d_{\lambda_1}^{-1},
\ee
and, for $k>0$,
\be\label{dlambda}
d_\lambda=\prod_{m'_i}(1-t^{m'_i+1})^{-1}=P_{t^{1/2}}(BH_\lambda),
\ee
where the exponents $m'_1,\ldots,m'_{{\rm rk} G}$ now denote the exponents of the group $H_\lambda$, the stabilizer of $e^{\lambda^\vee/k}$ (\textit{cf.} $\lambda$) in $K$.
\end{thm}

Alternatively, the expression that is more useful in proving this theorem is
$$
d_\lambda^{-1}=(1-t)^{\mathrm{rk}\,G}\sum_{w\in W_\lambda^{\mathrm{aff}}}t^{\ell(w)},
$$
where $W_\lambda^{\mathrm{aff}}=\mathrm{Stab}_{W_{\text{aff}}}(\lambda)$ is the stabilizer subgroup of $\lambda$ in the affine Weyl group associated with $G$. The reason that \eqref{dlambda} can be rewritten as above is the following. We first observe that $\mathrm{Stab}_{W_{\text{aff}}}(\lambda)$, generated by reflections along walls passing $\lambda$, is isomorphic to the Weyl group of $H_\lambda$. Then we use a well-known identity relating the sum over the Weyl group to a product involving the exponents of the Lie group, 
$$
\sum_{W_{H_\lambda}}t^{\ell(w)}=\prod_{m'_i}\frac{1-t^{m'_i+1}}{1-t}.
$$

Before proceeding to prove the above theorem, we first point out an immediate consequence. Just as in the case of ${\mathbb P}^1$ without marked points, combining Theorem~\ref{tqftp} with the general index formula, one obtains interesting identities involving the deformed characters $\Theta$. 
\begin{thm}\label{DC}
The deformed characters satisfy the following two orthogonality relations.
\begin{itemize}
\item The first orthogonality 
$$
\sum_{f\in F^{\mathrm{reg}}_{\rho,t}/W} \theta_t(f) \Theta_{\lambda_1,t}(f)\Theta_{\lambda_2,t}(f)=\delta_{\lambda_1,\lambda_2^*}(t)d_{\lambda_1}(t)^{-1}.
$$
\item The second orthogonality
$$
\sum_{\lambda\in \Lambda_k}\Theta_{\lambda,t}(f_i)\Theta_{\lambda^*,t}(f_j)\theta_t(f_i)=\delta_{ij}d_\lambda(t)^{-1},
$$
where $f_i,f_j\in F^{\mathrm{reg}}_{\rho,t}$.
\end{itemize}
\end{thm}
\begin{proof}
The first orthogonality relation can be obtained by computing the index for the moduli stack associated to ${\mathbb P}^1$ with two (maximal) parabolic points in two different ways,
\be\label{FirstO}
d_{0,2}(\lambda_1,\lambda_2)(t)=\sum_{f\in F^{\text{reg}}_{\rho,t}/W} \theta_t(f) \Theta_{\lambda_1,t}(f)\Theta_{\lambda_2,t}(f)=\delta_{\lambda_1,\lambda_2^*}{d_{\lambda_1}(t)^{-1}}.
\ee
The middle expression comes from the general form of the index formula in Theorem~\ref{thm:ParaVerlinde}, while the last expression is the result of Theorem~\ref{tqftp}, which will be proved by directly evaluating the index over this specific moduli stack.

The second orthogonality can be deduced from the first by considering the following expression,
$$
D=\sum_{\lambda_2}\sum_{f_i}\Theta_{\lambda_1,t}(f_i)\Theta_{\lambda_2,t}(f_i)\theta_t(f_i) \Theta_{\lambda_2^*,t}(f_j).
$$
Using the first orthogonality, we know that 
$$
D =\Theta_{\lambda_1,t}(f_j)d_{\lambda_1}(t)^{-1}.
$$
At $t=0$, the $\Theta$'s form a complete basis of functions on $F^{\text{reg}}_{\rho,t}/W$, and this property extends to a neighborhood of $t=0$. Then the completeness of $\Theta_{t}$'s necessarily implies that
$$
\sum_{\lambda}\Theta_{\lambda}(f_i)\Theta_{\lambda^*}(f_j)\theta_t(f_i)=\delta_{ij}d_\lambda^{-1}.
$$
\end{proof}

The rest of this section will be devoted to the proof of Theorem~\ref{tqftp}, carried out in several steps. Using the projection from the stack of parabolic bundles $\frak{P}_{0,2}$ to $\frak{M}(\mathbb{P}^1)$ as in the proof for Theorem~\ref{thm:MainP}, we can rewrite
$$
d_{0,2}=\mathrm{Ind}\left(\frak{M}(\mathbb{P}^1),S_tT_{\frak{M}{(\mathbb{P}^1)}}\otimes \frak{L}\otimes E_x(\Theta_{\lambda_1} \cdot\Theta_{\lambda_2})\right).
$$ 
We first analyze the contribution of the semi-stable stratum.

\begin{prop}\label{tqftpss} Over $\frak{M}^{\mathrm{ss}}(\mathbb{P}^1)=BG$, we have
$$
d_{0,2}^{\mathrm{ss}}\equiv \mathrm{Ind}\left(\frak{M}^{\text{ss}}(\mathbb{P}^1),S_tT \otimes \frak{L}\otimes E_x(\Theta_{\lambda_1,t} \Theta_{\lambda_2,t})\right)=\delta_{\lambda_1\lambda_2^*}\left(d_\lambda^{\mathrm{ss}}\right)^{-1},
$$
with
$$
(d_\lambda^\mathrm{ss})^{-1}=(1-t)^{\mathrm{rk}\,G}\sum_{w\in W_\lambda}t^{\ell(w)},
$$
where $W_\lambda=\mathrm{Stab}_{W}(\lambda)$.
\end{prop}

Notice the summation in the expression for $(d_\lambda^\mathrm{ss})^{-1}$ only runs over $W_\lambda$ (\textit{cf.} $W_\lambda^{\mathrm{aff}}$).

\begin{proof} Using $\frak{M}^{\text{ss}}=BG$, we have
$$
d_{0,2}^{\mathrm{ss}}=\mathrm{Ind}\left(BG,\lambda_{t}(\frak{g}[1])\otimes \Theta_{\lambda_1}\otimes \Theta_{\lambda_2}\right)=\dim\left(\lambda_{-t}\frak{g}\otimes \Theta_{\lambda_1}\otimes \Theta_{\lambda_2}\right)^G.
$$  
The last quantity can be written as an integral over $K$ 
$$
d_{0,2}^{\mathrm{ss}}=\dim\left(\lambda_{-t}\frak{g}\otimes \Theta_{\lambda_1}\otimes \Theta_{\lambda_2}\right)^G=\int_K [dg] \,{\det}_\frak{g}(1-t\mathrm{Ad}(g))\cdot\Theta_{\lambda_1}(g) \Theta_{\lambda_2}(g).
$$
Now one can prove the proposition by relating the deformed characters to the Hall-Littlewood polynomials, 
$$
\Theta_{\lambda}= \sum_{w\in W} \frac{(-1)^{l(w)} e^{w(\lambda+\rho)}}{\Delta \prod_{\alpha >0}(1-te^{w(\alpha)})}=C_{\lambda}\cdot\frac{\chi^{\text{HL}}_{\lambda}}{{\det}_{\frak{g}/\frak{t}} (1-t\mathrm{Ad})},
$$
where $C_\lambda$ is a constant given by 
$$
C_{\lambda}=\sum_{w\in W_\lambda}t^{\ell(w)}=\frac{(d_{\lambda}^{\text{ss}})^{-1}}{(1-t)^{{\rm rk}\,G}}.
$$
Then using orthonormal relations of Hall-Littlewood polynomials \cite{Mac}
$$
\int_G [dg] \frac{\chi^{\text{HL}}_{\lambda_1} \chi^{\text{HL}}_{\lambda_2}}{{\det}_\frak{g/\frak{t}}(1-t\mathrm{Ad})}=C_{\lambda_1}^{-1}\delta_{\lambda_1,\lambda_2^*}
$$
proves the proposition. 
\end{proof}

Notice that $d_\lambda^\text{ss}$ will only differs from $d_\lambda$ when $\lambda$ is on the affine wall of the Weyl alcove. To relate Proposition~\ref{tqftpss} and Theorem~\ref{tqftp} when $\lambda$ is not fixed by the affine reflection, we will prove a vanishing theorem concerning the contributions of unstable strata in the stack $\frak{M}(\mathbb{P}^1)$, which was also used in the proof for Proposition~\ref{g=0}. 

\begin{thm} \label{HigherVanish} Unstable strata in $\frak{M}(\mathbb{P}^1)$ will not contribute to the following indices. 
\begin{enumerate}
	\item 
	$$
	\mathrm{Ind}\left(\frak{M}(\mathbb{P}^1),S_tT_{\frak{M}}\otimes \frak{L}^k\right)
	$$
	when $k>2h-2$;
	\item 
	$$
\mathrm{Ind}(\frak{P}_{0,2},S_tT_{\frak{P}}\otimes\frak{L}_k)=\mathrm{Ind}\left(\frak{M}(\mathbb{P}^1),S_tT_{\frak{M}}\otimes \frak{L}^k\otimes E_x(\Theta_{\lambda_1}\cdot \Theta_{\lambda_2})\right)
$$
 when $k>0$ unless the two integrable weights are conjugate $\lambda_1=\lambda_2^*$ and live on the affine wall (\textit{i.e.}~$\langle\lambda_{1,2},\vartheta\rangle=k$ with $\vartheta$ being the highest root). 
\end{enumerate}
\end{thm}

\begin{proof} First, recall that the moduli stack of $G$-bundles $\frak{M}$ over $\Sigma$ admits an stratification \cite{Sh,AB}
$$
\frak{M}=\bigcup_{\xi}\frak{M}_\xi.
$$
Each stratum is labeled by a dominant co-character $\xi$ of $G$, and when $G$ is simply-connected, $\xi$ takes value in the co-root lattice. For $\Sigma=\mathbb{P}^1$, this stratification simply follows from Gothendieck's classification theorem. The index of a ``good'' K-theory class $\calE$ (in the sense that only finitely many strata contribute) over $\frak{M}$ can be decomposed into \cite{TW}
$$
\mathrm{Ind}(\frak{M},\calE)=\sum_{\xi}\mathrm{Ind}\left(\frak{M}^{\text{ss}}_{G_\xi,\xi}, \calE_\xi \otimes \mathbb{E}(\nu_\xi)_+^{-1}\right),
$$
where $\frak{M}^{\text{ss}}_{G_\xi,\xi}$ is the stack of semi-stable principal $G_{\xi}$-bundles of type $\xi$, $\calE_\xi $ is the restriction of $\calE$ to $\frak{M}^{\text{ss}}_{G_\xi,\xi}$, and $\mathbb{E}(\nu_\xi)_+^{-1}$ captures the contribution of the virtual normal bundle for the morphism $\frak{M}^{\text{ss}}_{G_\xi,\xi}\ra \frak{M}$. Denote the codimension of $\frak{M}_\xi$ as $d_\xi$, then from \cite[sec.~4.2]{TW} we have
$$
\mathbb{E}(\nu_\xi)_+^{-1}=\mathrm{Sym} E_x(\frak{g}/\frak{p}_\xi)^{2g-2}\otimes {\det}^{g-1} E_x(\frak{g}/\frak{p}_\xi)\otimes D(\frak{g}/\frak{p}_\xi)[d_\xi]
$$
where the last factor is the determinant line bundle associated with the $G_\xi$-representation $\frak{g}/\frak{p}_\xi$. 

When $\Sigma=\mathbb{P}^1$, $\frak{M}^{\text{ss}}_{G_\xi,\xi}=BG_\xi$, and the complex $\calE_\xi \otimes \mathbb{E}(\nu_\xi)_+^{-1}$
becomes a virtual $G_\xi$-representation. Therefore, computing the index over $\frak{M}^{\text{ss}}_{G_\xi,\xi}$ is equivalent to finding the formal dimension of the $G_\xi$-invariant subspace. Using the Weyl integral formula (or pushing forward to $BT$), we can convert this into a problem of finding invariant subspace of a $T$-representation, whose character is given by
\begin{align*}
\chi_\xi=&(1-t)^{\mathrm{rk}\,G}e^{k\langle\xi,\cdot\rangle}\prod_{\alpha>0}\frac{(1-te^{-\alpha})^{\alpha(\xi)+1}}{(1-te^{\alpha})^{\alpha(\xi)-1}}& \text{(from $S_tT\otimes\frak{L}^k$)} \\
& \times  (-1)^{d_\xi}\prod_{\alpha\in \frak{g}/\frak{p}_\xi} (1-e^{\alpha})^2 (e^\alpha)^{\alpha(\xi)-1} & \text{(from $\mathbb{E}(\nu_\xi)_+^{-1}$)}  \\ 
& \times\Theta_{\lambda_1}\Theta_{\lambda_2} & \text{(from deformed characters)}  \\
& \times |W_\xi|^{-1}\prod_{\alpha\in \frak{p}_\xi/\frak{b}} (1-e^{\alpha})^2 (-e^\alpha)^{-1} & \hspace{-3cm}\text{(from the Weyl integral formula),}
\end{align*}
where $|W_\xi|$ is the Weyl group of $G_\xi$. 
The expression for $\chi_\xi$ can be simplified to
\begin{align}\label{XiChar}
\chi_\xi=&|W_\xi|^{-1}(-1)^{d_\xi+\dim\frak{p}_{\xi}/\frak{b}}(1-t)^{\mathrm{rk}\,G}\prod_{\alpha>0}\frac{(1-te^{-\alpha})^{\alpha(\xi)+1}}{(1-te^{\alpha})^{\alpha(\xi)-1}}\\\nonumber
& \cdot e^{(k+h)\langle\xi,\cdot\rangle-2\rho}\prod_{\alpha>0} (1-e^{\alpha})^2 \Theta_{\lambda_1}\Theta_{\lambda_2}.
\end{align}
Without the deformed characters, there are only positive eigenvalues of $\xi$ when $k>2h-2$, with the smallest being
$$
k\langle\xi,\xi\rangle-4\rho(\xi)>(2h-2)\langle\xi,\xi\rangle-4\rho(\xi)=2\sum_{\alpha>0}\left(\alpha(\xi)^2-\alpha(\xi)\right)-2\langle\xi,\xi\rangle\ge0.
$$
The last inequality holds because there is at least one $\alpha$ such that $\alpha(\xi)\ge2$. And it is saturated by taking $\xi$ to be the dual of the highest root $\xi_\vartheta=\vartheta^\vee$, since
$$
\langle\rho,\vartheta\rangle=h-1.
$$
This proves the first part of the theorem, and we now proceed to prove the second part of theorem. We first rewrite the deformed characters as 
$$
\Theta_{\lambda}=\frac{\sum_{w\in W} (-1)^{l(w)} e^{w(\lambda+\rho)}\prod_{\alpha >0}\left(1-te^{-w(\alpha)}\right)}{e^{-\rho}\prod_{\alpha >0}(e^{\alpha}-1) \prod_{\alpha >0}(1-te^{\alpha})(1-te^{-\alpha})}.
$$
After cancellation, we obtain
\begin{align}\label{XiChar2}
&\chi_\xi=\frac{1}{|W_\xi|}(-1)^{d_\xi+\dim\frak{p}_{\xi}/\frak{b}}(1-t)^{\mathrm{rk}\,G}\prod_{\alpha>0}\frac{(1-te^{-\alpha})^{\alpha(\xi)-1}}{(1-te^{\alpha})^{\alpha(\xi)+1}}e^{(k+h)\langle\xi,\cdot\rangle} \times\\ \nonumber
 & \sum_{w_1,w_2\in W} (-1)^{l(w_1)+l(w_2)} e^{w_1(\lambda_1+\rho)+w_2(\lambda_2+\rho)}\prod_{\alpha >0}\left(1-te^{-w_1(\alpha)}\right)\left(1-te^{-w_2(\alpha)}\right).
\end{align}
We observe that the smallest eigenvalue of $\xi$ in $\chi_\xi$ is
$$
 k\langle\xi,\xi\rangle-(\lambda_1^*+\lambda_2^*)(\xi)\ge0,
$$
where we have used the fact that $\lambda$ and $-\lambda^*$ are related by the Weyl group.
Again, the second equality can be achieved only when $\xi=\xi_\vartheta$, for which
$$
k\langle\xi_\vartheta,\xi_\vartheta\rangle-(\lambda^*_1+\lambda^*_2)(\xi_\vartheta)=2k-\langle\lambda_1+\lambda_2,\vartheta\rangle\ge 0.
$$
Therefore, there is no $T$-invariant subspace unless both $\lambda_{1,2}$ live on the affine wall. Further, when both weights are on the affine wall, we denote $\frak{R}_0^+$ as the positive roots $\alpha$ such that $\langle\alpha,\vartheta\rangle=0$ (\textit{i.e.}~they are parallel to the affine wall), and $W_0\subset W$ are generated by reflections labeled by roots in $\frak{R}_0^+$. After eliminating terms in \eqref{XiChar2} that have non-zero $\xi_\vartheta$ eigenvalues, then the contribution given by the stratum $\frak{M}_{G_{\xi_\vartheta},\xi_\vartheta}$ is 
\begin{align*}
d_{0,2}^{\xi_\vartheta}(\lambda_1,\lambda_2)&\equiv \frac{1}{|W_0|}(-1)^{\dim\frak{g}/\frak{b}}(1-t)^{\mathrm{rk}\,G}(-t)\\ &\cdot\mathrm{Inv}\bigg[e^{(k+h-1)\vartheta}\prod_{\alpha\in \frak{R}_0^+}\frac{1}{(1-te^{-\alpha})(1-te^{\alpha})} \times  a_{\lambda_1}a_{\lambda_2} \times \\ \nonumber
\sum_{w_1,w_2\in W_0}& (-1)^{l(w_1)+l(w_2)} e^{-w_1(\lambda_1^*+\rho)-w_2(\lambda_2^*+\rho)}\prod_{\alpha \in \frak{R}_0^+}\left(1-te^{w_1(\alpha)}\right)\left(1-te^{w_2(\alpha)}\right)\bigg].
\end{align*}
Here $a_{\lambda}$ is a polynomial in $t$ associated to a dominant weight $\lambda$ on the affine wall defined in the following way,
$$
a_{\lambda} ={\sum_{w\in W_\lambda/W_{0,\lambda}}} t^{l(w)}
$$
where $w$ runs through minimal representatives of the cosets of $W_{\lambda,0}\equiv \mathrm{Stab}_{W_0}(\lambda)$. $a_\lambda$ arises because the $\xi_\vartheta$-invariant part of 
$$
\sum_{w\in W}(-1)^{l(w)}e^{\rho}\prod_{\alpha>0} \left(e^{-\frac12w(\alpha)}-te^{\frac12w(\alpha)}\right)
$$
is given by
$$
a_\lambda \left(\sum_{w\in W_0}(-1)^{l(w)}e^{\rho_0}\prod_{\alpha\in\frak{R}_0^+} (e^{-\frac12w(\alpha)}-te^{\frac12w(\alpha)})\right),
$$
since $l(w)=\#(\frak{R}_+\cap w^{-1}\frak{R}_+)$. Also, we have used the formula $d_\xi=2(\rho-\rho_\xi)(\xi)-\dim\frak{g}/\frak{p}_\xi$ to get 
$
d_\xi\equiv \dim\frak{g}/\frak{p}_\xi \pmod 2
$
in order to simplify the expression for $d_{0,2}^{\xi_\vartheta}$.
Define $\rho_0\equiv\sum_{\alpha\in\frak{R}_0^+}\alpha/2=\rho-(h-1)\vartheta/2$ and $\delta\lambda_i^*\equiv k\vartheta/2-\lambda_i^*$ for $i=1,2$. Although $\delta\lambda_i$ may not be a weight of $\frak{g}$, it is a weight of the root system $\frak{R}_0$, since $\langle\delta\lambda_i,\alpha\rangle\in \Z$ for all $\alpha\in \frak{R}_0$. From the orthogonality relations of the Hall-Littlewood polynomials associated with the root system $\frak{R}_0$, we have
\begin{align*}
(-1)^{\#(\frak{R}_0^+)}\delta_{\lambda_1,\lambda^*_2}\sum_{w\in W_{0,\lambda}}t^{l(w)}=\mathrm{Inv}\Bigg[
  \sum_{w_1,w_2\in W_0} (-1)^{l(w_1)+l(w_2)} e^{-w_1(\delta\lambda_1^*+\rho_0)-w_2(\delta\lambda_2^*+\rho_0)}\\ \times\prod_{\alpha \in \frak{R}_0^+}\frac{\left(1-te^{w_1(\alpha)}\right)\left(1-te^{w_2(\alpha)}\right)}{\left(1-te^{-\alpha}\right)\left(1-te^{\alpha}\right)}\Bigg]\cdot|W_0|^{-1},
\end{align*}
where $W_{0,\lambda}$ is the subgroup of $W_0$ that stabilize $\lambda$.
This completes the proof of the theorem and gives the contribution of the unstable strata when $\lambda_1=\lambda_2^*$ lives on the affine wall 
$$
d_{0,2}^{\xi_\vartheta}(\lambda_1,\lambda_2)= t(1-t)^{\mathrm{rk}\,G} a_{\lambda_1}^2\sum_{w\in W_{0,\lambda}}t^{l(w)}.
$$

\end{proof}

In fact, the bound $k>2h-2$ in the first part of the theorem cannot be maded stronger, as the stratum labeled by $\vartheta^\vee$ will always contribute when $k=2h-2$, with the contribution given by
$$
d_0^{\xi_\vartheta}\equiv\frac{(-1)^{\dim \frak{g}^+}}{|W_0|}(1-t)^{\mathrm{rk}\,G}\cdot\mathrm{Inv}\left(e^{(3h-2)\vartheta-2\rho}\prod_{\alpha>0}\frac{(1-e^{\alpha})^2(1-te^{-\alpha})^{\alpha(\xi_\vartheta)+1}}{(1-te^{\alpha})^{\alpha(\xi_\vartheta)-1}}\right).
$$
Eliminating mode non-invariant under $\xi_\vartheta$, and rewriting the second factor as 
$$
\mathrm{Inv}\left[e^{(2h-2)\vartheta-4\rho}\left(\prod_{\alpha>0,\atop \langle\alpha,\vartheta\rangle=0}(1-e^{\alpha})^2(e^{\alpha}-t)(1-te^{\alpha})\right)\left(\prod_{\alpha>0,\atop \langle\alpha,\vartheta\rangle=1}(t)^2\right)\cdot (-t)^3\right],
$$
enables us to simplify the expression for $d_0^{\xi_\vartheta}$,
$$
d_0^{\xi_\vartheta}=(1-t)^{\mathrm{rk}\,G}t^{2h-2+\#(\frak{R}^+\backslash \frak{R}_0^+)}\mathrm{Inv}\left(\prod_{\alpha\in \frak{R}_0}(1-e^{\alpha})(1-te^{\alpha})\right)/|W_0|
$$
where the last factor can be related to the Poincar\'e polynomial of $K_\vartheta$, the compact form of $[G_{\xi_\vartheta},G_{\xi_\vartheta}]$, via
\begin{align*}
&P_{-t}(K_\vartheta)=\int_{K_\vartheta} [dg]\det(1-t\mathrm{Ad})\\&=\frac{1}{|W_0|}(1-t)^{\mathrm{rk}K_\vartheta}\mathrm{Inv}\left(\prod_{\alpha\in \frak{R}_0^+}(1-e^{\alpha})(1-e^{-\alpha})(1-te^{-\alpha})(1-te^{\alpha})\right).
\end{align*}
This gives us
$$
d_0^{\xi_\vartheta}=(1-t)^{\mathrm{rk}\,G-\mathrm{rk}K_\vartheta}t^{2h-2+\#(\frak{R}^+\backslash \frak{R}_0^+)}P_{-t}(K_\vartheta).
$$
For $K=SU(2)$, $SU(3)$ and $G_2$, at level respectively $k=2,4,6$ we have
$$
d_{0,SU(2)}^{\xi_\theta}=t^2(1-t), \quad d_{0,SU(3)}^{\xi_\theta}=t^7(1-t)^2, \quad d_{0,G_2}^{\xi_\theta}=t^{11}(1-t)(1-t^3).
$$

For the second part of Theorem~\ref{HigherVanish}, notice that the statement would not be true if we had replaced the \emph{maximal} deformed characters $\Theta_{\lambda}\equiv\Theta_{\lambda,B}$ with $\Theta_{\lambda,P'}$ associated to a more general parabolic structure $P'\supset B$. This is part of the reason that the basis given by maximal parabolic structures will be better behaved in the 2d TQFT (compared to \textit{e.g.}~the basis given by $\Theta_{\lambda,P_\lambda}$ with $P_\lambda$ being the stabilizer of $\lambda$ in $G$).  

We now complete the proof for Theorem~\ref{tqftp}, by adding the contribution of unstable strata to $d_{0,2}^{\text{ss}}$ when $\lambda_1=\lambda_2^*$ is fixed by the affine reflection.

\begin{proof}[Proof for Theorem~\ref{tqftp}] As we have shown, beside $\frak{M}^\text{ss}$ only the stratum labeled by $\xi_\vartheta$ could contribute even if $\lambda=\lambda_1=\lambda_2^*$ lives on the affine wall. Incorporating its contribution leads to
$$
d_{0,2}(\lambda,\lambda^*)=d_{0,2}^{\text{ss}}+t(1-t)^{\mathrm{rk}\,G} a_{\lambda}^2\sum_{w\in W_{0,\lambda}}t^{l(w)}.
$$
To prove the theorem, we only need to show
\be\label{CombInd}
\sum_{w\in W_\lambda^{\text{aff}}}t^{l(w)}=\sum_{w\in W_\lambda}t^{l(w)}+t a_{\lambda}^2\sum_{w\in W_{0,\lambda}}t^{l(w)}.
\ee
This follows from the argument below. 

$W_\lambda^{\text{aff}}$ is obtained from $ W_\lambda$ by adding one element $r_\vartheta$. Any elements in $W_\lambda^{\text{aff}}\backslash W_\lambda$ can be written as $w_1r_\vartheta w_2$ with $w_1,w_2\in W_\lambda$. In fact, as the subgroup $W_{\lambda,0}$ commutes with $r_\vartheta$, any elements can be written as
$w_1'r_\vartheta w_0 w'_2$, where $w_1',w_2'$ are minimal representatives of cosets $W_\lambda/W_{\lambda,0}$ and $w_0\in W_0$, and there is no further redundancy. Then, breaking the right-hand side of the equation
$$
\sum_{w\in W_\lambda^\text{aff}}t^{l(w)}=\mathrm{Inv}\left(\sum_{w\in W^{\text{aff}}}(-1)^{l(w)}e^{\rho-w(\rho)}\prod_{\alpha>0} (1-te^{w(\alpha)})\right)
$$
into a sum over $W_\lambda$ and $W_\lambda^{\text{aff}}\backslash W_\lambda$ proves \eqref{CombInd}.

\end{proof}

\section{The Verlinde algebra and the 2D TQFT viewpoint} \label{tqft}
In this section, we will study the 2D TQFT underlying the index formula. The existence of a TQFT structure from the index theory of $\frak{M}$ is foreseen in \cite{Te} and highlighted in \cite{GP} and \cite{GPYY}. We will denote the TQFT functor $Z(G,k,t)$ and simply denote it $Z$ when no confusion will be caused.

The functor $Z$ associates to a circle $S^1$ the vector space
$$
Z(S^1)= {\mathcal V}_{G}^{(k)}.
$$
The 2D part of the theory is simply given by the index, \textit{e.g.}~if $X$ is a smooth surface of genus $g$ with $n$ boundary components, we define 
\be
Z(X) : Z(S^1)^{\otimes n} \ra {\mathbb C}
\ee
by the assigment
\be
Z(X)(\vec{\lambda})=d_{g,n}(\vec\lambda),
\ee 
where $\vec \lambda$ is any assignment of labels, \textit{e.g.}~basis elements of ${\mathcal V}_{G}^{(k)}$, to the $n$ boundaries of $X$.

In order to prove that this $Z$ really is a 2D TQFT, we need now to establish the TQFT gluing rules, which is equivalent to proving Theorem \ref{thm:Main3}. However, Theorem \ref{thm:Main3} immediately follows from the second orthogonality relation of the deformed characters in Theorem \ref{DC}.

We observe that the dimension of $Z(S^1)$ does not depend on $t$ and equals the size of the index set 
$$
Z(T^2)=|F_{\rho,t}^{\text{reg}}/W| = |F_{\rho}^{\text{reg}}/W|.
$$

Let us now briefly recall how one induces the structure of a Frobenius algebra on $Z(T^2)$ using the 2D TQFT Z. Let $T$ be a surface of genus zero with three boundary components. We define the product
$$ \lambda \star_t \mu = \sum_{\nu} Z(T)( \lambda, \mu, \nu^\dagger) \nu,$$
where $\nu^\dagger \in Z(T^2)$ is the unique element such that
$$ Z(C)(\lambda, \nu^\dagger) = \delta_{\lambda, \nu},$$
where $C$ is a genus-zero surface with two boundary component, which is also used to induce the needed bilinear pairing
$$  \langle \lambda,\mu\rangle_t = Z(C)( \lambda,\mu).$$
By Proposition \ref{tqftp}, the product is of course  the same as the product on ${\mathcal V}_{G}^{(k)}$ from Definition \ref{VA}. 

\proof[Proof of Theorem \ref{thm:Main4}]
First we check that $\star_t$ is associative, so we compute
$$ (\lambda \star_t \mu)\star_t \kappa =\sum_{\rho\in \Lambda_k} \sum_{\nu\in \Lambda_k}  Z(T)( \lambda, \mu, \nu^\dagger)(t) Z(T)( \nu, \kappa, \rho^\dagger)(t) \rho,$$
and
$$ \lambda \star_t (\mu\star_t \kappa) =\sum_{\rho\in \Lambda_k} \sum_{\nu\in \Lambda_k}  Z(T)( \lambda, \nu, \rho^\dagger)(t) Z(T)( \mu, \kappa, \nu^\dagger)(t) \rho.$$
If now $S$ is a surface of genus zero with four boundaries then by factoring $S$ in two different ways, we get that
$$Z(S)(\lambda,\mu, \kappa, \rho^\dagger) = \sum_{\nu\in \Lambda_k}  Z(T)( \lambda, \mu, \nu^\dagger)(t) Z(T)( \nu, \kappa, \rho^\dagger)(t)$$
and
$$Z(S)(\lambda,\mu, \kappa, \rho^\dagger) = \sum_{\nu\in \Lambda_k}  Z(T)( \mu, \kappa, \nu^\dagger)(t) Z(T)( \lambda, \nu, \rho^\dagger)(t)$$
by definition of $(\cdot)^\dagger$, Proposition \ref{tqftp} and Theorem \ref{thm:Main3}. From these two equations the associativity follow. Next we establish that the compatibility between the bilinear pairing and the product
$$ \langle  \lambda \star_t \nu ,\mu\rangle_t  =  \langle \lambda, \nu \star_t\mu\rangle_t,$$
by observing that they are both equal to
$$ \sum_{\rho} Z(T)( \lambda, \mu, \rho^\dagger)Z(C)(\rho, \nu)=Z(T)(\lambda,\mu,\nu) = \sum_{\rho} Z(T)( \nu, \mu, \rho^\dagger)Z(C)(\lambda,\rho).
$$
\eproof

\section{The 2D TQFT for $SU(2)$}

In the rest of the paper, we will describe explicitly the TQFT for $K=SU(2)$ and give many examples for the index formula associated to various (parabolic) moduli spaces of Higgs bundles.

For $K=SU(2)$ at positive level $k$, $V$ is now spanned by say $w_\lambda$, indexed by the integrable weights $\lambda=0,1,2,\ldots,k$. A 2D TQFT is determined by its value on the three-punctured sphere and the cylinder, or equivalently the ``fusion coefficients'' $f\in V^{*\otimes 3}$ and the ``metric'' $\eta\in V^{*\otimes 2}$. 

\begin{thm}
The 2D TQFT associated with $SL(2,\C)$-Higgs bundles is given by $V$ and the following data. 
\be\label{Metric}
\eta^{\lambda_1\lambda_2}=\mathrm{diag}\{1-t^2,1-t,\ldots,1-t,1-t^2\},
\ee
and
\be\label{Fusion}
f^{\lambda_1\lambda_2\lambda_3}=\left\{\begin{array}{ll} 1 & \textrm{if $\lambda_1+\lambda_2+\lambda_3$ is even and $\Delta\lambda\leq 0$,}\\
t^{\Delta\lambda/2} & \textrm{if $\lambda_1+\lambda_2+\lambda_3$ is even and $\Delta\lambda> 0$,}\\
0 & \textrm{if $\lambda_1+\lambda_2+\lambda_3$ is odd.}\end{array}\right. 
\ee
Here, $\Delta\lambda=\mathrm{max}(d_0,d_1,d_2,d_3)$ with
\begin{align*}
d_0&=\lambda_1+\lambda_2+\lambda_3-2k,\\
d_1&=\lambda_1-\lambda_2-\lambda_3,\\
d_2&=\lambda_2-\lambda_3-\lambda_1,\\
d_3&=\lambda_3-\lambda_1-\lambda_2.
\end{align*}

\end{thm}
The above TQFT rules first appeared in \cite{GP}, following from an intuitive geometric argument, and later supported by physics computation in \cite{GPYY}. Now, we can give it a proof simply by specializing formulae for general $G$ to $SL(2,\C)$.
\begin{proof}
The expression for $\eta$ directly follows from Proposition \ref{tqftp}, and it will be verified in Appendix \ref{FusionCoef} that $f$ is indeed given by the above expression.

\end{proof}

The fusion coefficients $f$ and the metric $\eta$ are related by the ``cap state''
$$
w_{\emptyset}=w_0-t w_2
$$
via
$$
\eta^{\lambda_1\lambda_2}=f^{\lambda_1\lambda_2\emptyset}=f^{\lambda_1\lambda_20}-t f^{\lambda_1\lambda_22}.
$$
Clearly, $f$ and $\eta$ are symmetric. The metric $\eta$ defines a dual basis of $w_\lambda$ for $V^*$, which we will denote by $w^{\lambda}$. It also gives a metric on $V^*$, which we will denote by the same symbol and write $\eta_{\lambda_1\lambda_2}$ in the $\{w^\lambda\}$ basis. With the identification of $V$ and $V^*$ using $\eta$, $f$ can now be used to define a product on $V$, whose coefficients  in the basis $w_\lambda$ we denote $f^{\lambda_1\lambda_2}_{\phantom{\lambda_1\lambda_2}\lambda_3}$.

The associativity of the algebra follows from Theorem \ref{thm:Main5}. This ensures that $d_{0,4}\in V^{*\otimes 4}$ associated with a fourth-punctured $\mathbb{P}^1$ is well-defined,
$$
d_{0,4}^{\lambda_1\lambda_2\lambda_3\lambda_4}=f^{\lambda_1\lambda_2\mu}\eta_{\mu\nu}f^{\nu\lambda_3\lambda_4}=f^{\lambda_1\lambda_3\mu}\eta_{\mu\nu}f^{\nu\lambda_2\lambda_4}.
$$
The explicit expression for $d_{0,4}$ is 
\be\label{FourPuncture}
d_{0,4}^{\lambda_1\lambda_2\lambda_3\lambda_4}=t^{kh_0}\left(\frac{\tilde{d}_{0,4}}{1-t}+\frac{2t}{1-t^2}\right)+\frac{\sum_i t^{kh_i}}{(1-t^{-1})(1-t^2)}.
\ee
Here all $h_0,\ldots,h_4$ and $\tilde{d}_{0,4}$ are functions of the $\lambda$'s whose explicit form will be given in Appendix \ref{FusionCoef}. In particular,
$$
\tilde{d}_{0,4}=\lim_{t\ra 0}\left(t^{-kh_0} d_{0,4}\right)
$$
gives the (undeformed) Verlinde formula when $h_0=0$. 

The formula \eqref{FourPuncture} can also be interpreted as the $\C^*$-equivariant Atiyah-Bott localization formula applied to the moduli space of Higgs bundles. Recall that the Hitchin moduli space in this case is an elliptic surface. The nilpotent cone $\CN$ of the Hitchin fibration is the only singular fiber with affine $D_4$ singularity (or, equivalently, of Kodaira type $I_0^*$). It has five components
\be
\CN=M\cup_{i=1}^4 D_i.
\ee
And the first term in \eqref{FourPuncture} comes from $M$ where the moment map of the $\C^*$-action is $h_0$, and the last comes from the higher fixed points where the moment map evaluates to be $h_i$.\footnote{The normalization we used for the moment map is related to Hitchin's in \cite{Hit} by a factor of $\pi$. As a consequence, our moment map is always rational at the critical points.}  Indeed, $M$ is $\cp^1$ and its contribution in the localization formula is 
$$
\mathrm{Ind}(\cp^1,\CL\otimes S_t O(-2))=\sum_{i=0}^{\infty}t^i\left(\mathrm{Ind}(\cp^1,\CL)-2n\right),
$$
where $\CL$ is the restriction of $L^k$ to $M$. After summation, this indeed agrees with the first term in \eqref{FourPuncture}, given that the degree of $\CL$ is $\tilde{d}_{0,4}-1$ when $h_0=0$ and $-\tilde{d}_{0,4}+1$ when $h_0\neq 0$.

As the next example, we move to the genus-three case. From the TQFT rule, one obtains\footnote{Computations via functoriality of TQFT are obtained using Mathematica for the next three examples, and we are very grateful for Ke Ye for his extensive help.} 
\begin{equation}
\begin{aligned}
d_3=\frac{1}{(1-t)^6}& \left(b_6 k^6 +b_5 k^5 +b_4 k^4 + b_3 k^3 + b_2 k^2 + b_1 k +b_0 \right)
\end{aligned}
\end{equation}
where 
\begin{equation*}
\begin{aligned}
b_6  =& \frac{1} {180 }\\
b_5  =& \frac{ (1+t)}{15(1-t)}\\
b_4  =& \frac{ \left(7 t^2-2 t+7\right)}{18 (1-t)^2}\\
b_3  =& \frac{4  \left(t^6-t^5-4 t^4-10 t^3-4 t^2-t+1\right)}{3 (1-t^2)^3}\\
b_2  =& \frac{1}{180 (1-t^2)^4} \big(31680 t^{4 + k} + 469 t^8-2280 t^7+44 t^6-6360 t^5\\
&\quad \quad\quad\quad \quad +7614 t^4-6360 t^3+44 t^2-2280 t+469\big)\\
b_1  =& \frac{ 1}{5(1-t^2)^5} \big(4160 t^{4 + k} + 3200 t^{5 + k} + 4160 t^{6 + k}+13 t^{10}-114 t^9+361 t^8\\
&\quad \quad\quad\quad \quad +296 t^7+2986 t^6+1556 t^5+2986 t^4+296 t^3+361 t^2-114 t+13\big)\\
b_0  =& \frac{1}{(1-t^2)^6} \big( 960 t^{4 + k} + 1536 t^{5 + k} + 2944 t^{6 + k} + 1536 t^{7 + k} + 
 960 t^{8 + k} + 64 t^{6 + 2 k} \\
&\quad \quad\quad\quad \quad+ t^{12}-12 t^{11}  +66 t^{10}-220 t^9-465 t^8-2328 t^7-2084 t^6-2328 t^5\\
&\quad \quad\quad\quad \quad-465 t^4-220 t^3+66 t^2-12 t+1\big)
\end{aligned}
\end{equation*}
The term involving $t^{2k}$ can be written as
$$
\frac{64 t^{2k}}{\left(1-t^{-1}\right)^6\left(1-t^{2}\right)^6}
$$
and agrees with the existence of 64 critical points at the value $2$ of moment map $h$ whose 12-dimensional normal bundle splits as $\bbC^6[-1]\oplus\bbC^6[2]$.\footnote{Here the number in the parenthesis represent the eigenvalue of the $\C^*$ action.} And the terms proportional to $t^{k}$ are related to the existence of a two-dimensional critical manifold at $h=1$. 

For $K=SU(2)$, the genus-two moduli space is especially interesting, due to the non-trivial role played by the strictly semi-stable loci. The Verlinde formula for Higgs bundles in this case is given by
\begin{equation}\label{g=2}
d_2 = \frac{1}{(1-t)^3} (c_3 k^3+c_2 k^2 +c_1 k +c_0)
\end{equation}
where
\begin{equation*}
\begin{aligned}
& c_3 = \frac{1}{6},\\[0.5em]
& c_2 = \frac{1+t^2}{1-t^2},\\[0.5em]
& c_1 = \frac{11-36t-9t^2+9t^4+36t^5-11t^6}{6(1-t^2)^3},\\[0.5em]
& c_0 = \frac{1-6t+15t^2-4t^3+15t^4-6t^5+t^6-16t^{k+3}}{\left(1-t^2\right)^3}.
\end{aligned}
\end{equation*}
The term proportional to $t^k$ can be written as
$$
\frac{16t^{k}}{(1-t^{-1})^3(1-t^2)^3}
$$
and could be identified with the contribution from the 16 higher fixed points of the $C^*$-action. The rest in \eqref{g=2} should be coming from the bottom component $M_0=\cp^3$. The normal bundle of $M_0$ in the Higgs bundle moduli space is not $T^*\cp^3$ any more. The correction from the additional divisor is given by
\begin{multline}
d_2-\mathrm{Ind}(\cp^3,S_tT\otimes \CL^k)-\frac{16t^{k}}{(1-t^{-1})^3(1-t^2)^3}=\\
\frac{1}{(1-t)^3}\left[\frac{2t}{1-t^2}k^2+\frac{18t}{(1-t)(1-t^2)}k+\frac{18t(1+t^2)}{(1-t)(1-t^2)^2}-\frac{12t^2}{(1-t)(1-t^2)}\right].
\end{multline}

Moduli spaces of Higgs bundles often has singularities. This is however not the case in the co-prime cases, where the underlying principal $G$-bundle is twisted by a line bundle. As was found in \cite{GP} and \cite{GPYY}, twisting by a line bundle also has a nice interpretation in the TQFT language. For $SU(2)$, twisting by an odd line bundle is equivalent to first adding a puncture and then attaching a  ``twisted-cap'' that is associated with the following vector in $V$, 
\begin{equation}
w_\psi=w_k-t w_{k-2}.
\end{equation}
This relation follows from an observation about the decomposition of the function $(-1)^j$ on $F_t/W$,
\be
(-1)^j=\Theta_{k}-\Theta_{k-2},
\ee
where $j$ orders the $k+1$ points in $F_t/W$ by their distance to the identity of $SU(2)$.

Now we will test this conjecture by computing the index for $g=2$ and compare with the localization formula. The index for even $k$ is given by
\begin{equation}\label{GenusTwoMinusOne}
\begin{aligned}
d_2^{\text{odd}}  = \frac{1}{12 (1-t)^6 (1+t)^3} (a_3 k^3 + a_2 k^2 + a_1 k + a_0)
\end{aligned}
\end{equation}
where
\begin{equation*}
\begin{aligned}
a_3 & = \left(1- t^2\right)^3\\[0.5em] 
a_2 & = 6 (1-t)^2 (1+t)^4\\[0.5em]
a_1 & = 2 \left(1 - t^2\right) \left(144 t^{\frac{k}{2}+2}+7 t^4+12 t^3+10 t^2+12 t+7\right)\\[0.5em]
a_0 & = 12 \left(48 t^{\frac{k}{2}+2}+32 t^{\frac{k}{2}+3}+48 t^{\frac{k}{2}+4}+t^6-6 t^5-33 t^4-52 t^3-33 t^2-6 t+1\right).
\end{aligned}
\end{equation*}
When $k$ is odd, the index is zero.\footnote{In fact, in the convention that we are using, $\CL^k$ is only well-defined over the moduli space for even $k$. The convention commonly used in the literature is related to ours by a factor of 2.}

Now we show that it agrees with direct computation using the localization formula. The fixed points under the $\C^*$-action are studied in \cite{Hit} and there are two components $M_0$ and $M_1$. The first is the moduli space of flat connections $M$ sitting at $\mu=0$. The cohomology of $M$ is generated by $\alpha\in H^2(M,\Z)$, $\psi_1,\psi_2,\psi_3,\psi_4\in H^3(M,Z)$ and $\beta\in H^4(M,\Z)$. The non-zero intersection numbers are  \cite{N, KN}
$$
\int_M\alpha\^ \beta=-4, \quad \int_M\alpha^3=4,\quad \int_M\gamma=\int_M\left(\sum_{i=1}^4\psi_i\right)^2=4.
$$
And the total Chern character of $TM$ is given by
$$
\mathrm{ch}(TM)=3+2\alpha+\beta+\frac{1}{3}\alpha\beta-\frac{4}{3}\gamma.
$$
Then one can explicitly compute the index for $S^nTM\otimes\CL^k$, and a perfect match is found with \eqref{GenusTwoMinusOne}. For example, for $n=1$,
\begin{eqnarray*}
\chi(TM\otimes\CL^k)&=&\int_M e^{(k/2+1)\alpha}(3+2\alpha+\beta+\frac{1}{3}\alpha\beta-\frac{4}{3}\gamma)\left( \frac{\sqrt{\beta}/2}{\sinh \sqrt{\beta}/2}\right)^2\\
& =& \frac{1}{4}\left(k^3+10 k^2+22 k-12\right).
\end{eqnarray*}
For $n=2$, with the help of the splitting principle, one finds
$$
\mathrm{ch}(S^2TM)=6+8\alpha+5\beta+\frac{13}{3}\alpha\beta-\frac{28}{3}\gamma+2\alpha^2.
$$
One can easily compute the index to be 
$$
\chi(S^2TM\otimes\CL^k)=\frac{1}{2}\left(k^3+14 k^2+34 k-84\right).
$$

One can compute the index to all orders of $t$ by expressing $\mathrm{ch}(S_t TM)$ in terms of $\alpha$, $\beta$ and $\gamma$. For the purpose of computing the index, one can make the substitution $\beta\rightarrow -\alpha^2$ and $\gamma\rightarrow \alpha^3$. Then after a short computation, one finds
$$
\mathrm{ch}(S_t TM) \sim \left[(1-t)^3-2\alpha t(1-t)^2+\alpha^2 t(1-t)(1+2t)-\alpha^3 t(1+3t+\frac{4}{3}t^2)\right]^{-1}.
$$
This enables us to obtain a formula for the index of $S_t TM\otimes\CL^k$ that is valid to all order of $t$
\be\label{GenusTwoMinusOneIndex}
\chi(S_t TM\otimes\CL^k)=\int_M e^{(k/2)\alpha}\^\mathrm{ch}(S_t TM)\^\mathrm{Td}(TM)\\=\frac{1}{(1-t^3)}(a'_3 k^3+a'_2 k^2+a'_1 k+a'_0),
\ee
where
\begin{equation}
\begin{aligned}
a'_3 & = \frac{1}{12}\\[0.5em] 
a'_2 & = \frac{1+t}{2(1-t)}\\[0.5em]
a'_1 & = \frac{7-2t+7t^2}{6(1-t)^2}\\[0.5em]
a'_0 & = \frac{(1+t)(3+2t+3t^2)}{3(1-t)^3}.
\end{aligned}
\end{equation}
Notice that the expression above is simply \eqref{GenusTwoMinusOne} without all terms proportional to $t^{k/2}$. These extra terms
\be\label{GenusTwoHigher}
t^{k/2+2}\left[\frac{24k}{(1-t)^5(1+t)^2}+\frac{16(3+2t+3t^2)}{(1-t)^6(1+t)^3}\right]
\ee
on one hand comes from the higher sheaf cohomology groups of $S_t TM\otimes\CL^k$, and, on the other, is the contribution from the higher critical manifold $M_1$. This formula contains much information about $M_1$ and its normal bundle. And it is straightforward to derive this expression from the following facts. $M_1$ is a 16-fold covering of $\Sigma$ and its volume is 24. The normal bundle splits as $L_{-1}^{\oplus 2}\oplus L_1\oplus L_2^{\oplus 2}$, where the subscript stands for the eigenvalue of the $U(1)$-action and their degrees are given by $32,32$ and $-32$. Then the localization formula gives the contribution of $M_1$ as the following integral 
$$
e^{k/2}\int_{M_1}\frac{\Td(M_1)\^ e^{24 k x}}{(1-t^{-1}e^{-32x})^2(1-te^{-32x})(1-t^{2}e^{+32x})^2}
$$
where $x$ represent the normalized volume form with $\int_{M_1}x=1$. It is straightforward to evaluate this expression, reproducing \eqref{GenusTwoHigher}.

\appendix


\section{Index Formula Made Explicit}\label{sec:Diag}

In \cite{GP}, the ``equivariant Verlinde formula'' was written down in physics language for $K=SU(N)$ using a quantity called a ``twisted effective superpotential'' $\tilde{W}$.\footnote{It was in fact for $U(N)$ along with a method to convert it into $SU(N)$. All formulae quoted from \cite{GP} in this section is after the conversion.} It is a function on $\frak{t}$ and after choosing $T_K\subset SU(N)$ to be parametrized by diagonal matrices with the $i$-th diagonal entry being $e^{2 \pi i\sigma}$, $\tilde{W}$ can be expressed as 
$$
\tilde{W}(\sigma)=\pi i(k+N)\sum_{a=1}^N \sigma_a^2+\frac{1}{2\pi i}\sum_{a\neq b}\mathrm{Li}_2\left[te^{2\pi i(\sigma_a-\sigma_b)}\right]+\pi i\sum_{a>b}(\sigma_a-\sigma_b).
$$
Up to an additive constant, $-2\pi i \tilde{W}$ is exactly the function $D_t$ on $T$ in \eqref{FunD} with $\xi=2\pi i$. The ``Bethe ansatz equations'' are given by
\be\label{BAE}
\exp\left[\frac{\partial \tilde{W}}{\partial (\sigma_a-\sigma_b)}\right]=1,\quad \text{ for all $a,b=1,2,\dots, N$.}
\ee
Using
$$
\frac{\partial \mathrm{Li}_2(e^x)}{\partial x}=-\ln(1-e^x),
$$
it is straightforward to check that \eqref{BAE} is exactly the equation $\chi'_t(f)=e^{2\pi i\rho}$. From \cite{GP} we have the following expression for the equivariant Verlinde formula 
\begin{multline}\label{EVF}
d_g=\sum_{\sigma\in F^{\text{reg}}_{\rho,t}} \left(N(1-t)^{(N-1)(1-R)}\det\left|\frac{1}{2\pi i}\frac{\partial^2 \tilde{W}}{\partial \sigma_a\partial\sigma_b}\right|\frac{\prod_{\alpha}(1-te^{\alpha(\sigma)})^{1-R}}{\prod_{\alpha}(1-e^{\alpha(\sigma)})} \right)^{g-1}.
\end{multline}
As $N \det\left|\frac{1}{2\pi i}\frac{\partial^2 \tilde{W}}{\partial \sigma_a\partial\sigma_b}\right|$ is exactly the same as $|F|\det(H^\dagger_t)$, \eqref{EVF} agrees with the index formula in Theorem \ref{thm:Main5} that we have proved in this paper.

For $K=SU(2)$, we parametrize the Cartan subgroup $U(1)$ by $e^{i\vartheta}$. The index formula is then given by
$$
Z=\sum_{\vartheta}\theta_t(e^{i\vartheta})^{1-g},
$$
where we are summing over the solutions to the equation
\be\label{BAESU2}
e^{2i(k+2)\vartheta}\left(\frac{1-te^{-2i\vartheta}}{1-te^{2 i \vartheta}}\right)^2=1
\ee
in $(0,\pi)$ and 
$$
\theta_t(e^{i\vartheta})= (1-t)^{R-1}\frac{4 (\sin\vartheta)^2 \cdot |1-t e^{2i\vartheta}|^{2R-2}}{|F|\det H^\dagger},
$$
with  
$$
|F|\det H^\dagger= 4\cdot\left[\frac{k+2}{2}+\frac{2t\cos2\vartheta -2t^2}{|1-t e^{2i\vartheta}|^2}\right].
$$
For $R=2$, the index indeed agrees with (7.51) in \cite{GP}, and for $R=0$, it gives the index for the total lambda class of the cotangent bundle to $\frak{M}$ as in \cite{TW} and \cite{TeNote} (see also \cite{OY} where a quantum field theory approach was used).

The indices for the parabolic moduli spaces are given by
$$
d_{g,n}(\{\lambda_i\})=\sum_{f\in F^{\text{reg}}_{\rho,t}/W}\theta_t(f)^{1-g}\prod_i\Theta_{t,\lambda_i}(f),
$$
and for $SU(2)$ we have
\be\label{DeformedSU2}
\Theta_\lambda(e^{i\vartheta})=\frac{\sin\left[(\lambda+1)\vartheta\right]-t\sin\left[(\lambda-1)\vartheta\right]}{\sin\vartheta|1-t e^{2i\vartheta}|^2}.
\ee
This agrees with (7.47) of \cite{GP} up to normalization constants.\footnote{The deformed characters can be viewed as columns in the transfer matrix between two natural bases of the 2D TQFT Hilbert space---the ``diagonal basis'' and the ``parabolic basis''. The normalization constants in \cite{GP} was chosen such that all the basis vectors has norm 1.} In particular, 
$$
\Theta_0=\frac{(1+t)}{|1-t e^{2i\vartheta}|^2},
$$
and the constant function on $T$ can be decomposed as
$$
1=\Theta_0-t\Theta_2,
$$
in agreement with the decomposition of the ``cap state'' in the 2D TQFT,
$$
e_{\emptyset}=e_0-te_2.
$$ 

\section{The $SU(2)$ TQFT}\label{FusionCoef}
In this appendix, we will prove that the structure constants $f^{\lambda_1\lambda_2\lambda_3}$ in \eqref{Fusion} give the Verlinde algebra for $SL(2,\C)$-Higgs bundles. Given the completeness of the deformed characters $\Theta_\lambda$ on the space of functions on $F_{\rho,t}^{\text{reg}}/W$ when $t$ is small, we only need to prove the following proposition.
  
\begin{prop} 
The deformed characters on $F_{\rho,t}^{\text{reg}}$ for $G=SU(2)$ satisfy the identity 
\be\label{FusionIdentity}
\Theta_{\lambda_1}\Theta_{\lambda_2}=\sum_{\nu\in \Lambda_k} f^{\lambda_1\lambda_2\nu}d_{\nu}\Theta_{\nu}
\ee
for all $\lambda_1$ and $\lambda_2$ at all positive level $k$. 
\end{prop}

\begin{proof}
If we define
$$
\delta_1=|\lambda_1-\lambda_2|,\quad \delta_2=\mathrm{min}(\lambda_1+\lambda_2,2k-\lambda_1-\lambda_2)
$$
one can rewrite \eqref{Fusion} as 
$$
f^{\lambda_1\lambda_2\nu}=\left\{\begin{array}{ll}
t^{(\delta_1-\nu)/2} & \textrm{if $\nu\leq\delta_1$,}\\
1 & \textrm{if $\delta_1\leq\nu\leq\delta_2$, }\\
t^{(\nu-\delta_2)/2} & \textrm{if $\nu\geq\delta_2$,}\end{array}
\right.
$$
when $\lambda_1+\lambda_2+\nu$ is even.

We first look at the case with $\lambda_1+\lambda_2$ being odd and $k$ being even. Using the explicit form \eqref{DeformedSU2} of $\Theta_\nu$ for $SU(2)$, the right-hand side of \eqref{FusionIdentity} can be decomposed into three parts
$$
\sum_{\nu\in \Lambda_k} f^{\lambda_1\lambda_2\nu}d_{\nu}\Theta_{\nu}=\frac{1}{(1-t)\sin\vartheta|1-te^{2i\vartheta}|^2}(X+Y+W)
$$
with\footnote{When $\delta_2=\delta_1+2$, $Y=0$. And when $\delta_2=\delta_1$, 
$$
Y=-\left[\sin(\delta_1+1)\vartheta-t\sin(\delta_1-1)\vartheta\right].
$$
The algebraic manipulation we will perform automatically take into account these special situations.
 }
\begin{align}\nonumber
X=&\sum_{\nu=1}^{\delta_1} t^{(\delta_1-\nu)/2}\cdot \left[\sin(\nu+1)\vartheta-t\sin(\nu-1)\vartheta\right],\\ \label{ABCSum}
Y=&\sum_{\nu=\delta_1+2}^{\delta_2-2}\left[\sin(\nu+1)\vartheta-t\sin(\nu-1)\vartheta\right],\\ \nonumber
W=&\sum_{\nu=\delta_2}^{k-1}t^{(\nu-\delta_2)/2}\left[\sin(\nu+1)\vartheta-t\sin(\nu-1)\vartheta\right],
\end{align}
where $\nu$ takes value in odd integers in all summations.

Rewriting the sine functions using exponentials and summing the geometric series give
\begin{align} \nonumber
X=&\sin(\delta_1+1)\vartheta,\\\label{ABCResult}
Y=&\frac{1}{2i}\frac{1-te^{-2i\vartheta}}{1-e^{2i\vartheta}}\cdot\left(e^{i(\delta_1+3)\vartheta}-e^{i(\delta_2+1)\vartheta}\right)+\text{c.c.},\\ \nonumber
W=&\frac{1}{2i}\frac{1-te^{-2i\vartheta}}{1-te^{2i\vartheta}}\cdot\left(e^{i(\delta_2+1)\vartheta}-t^{(k+1-\delta_2)/2}e^{i(k+2)\vartheta}\right)+\text{c.c.}\\ \nonumber
=&\frac{1}{2i}\frac{1-te^{-2i\vartheta}}{1-te^{2i\vartheta}}\cdot e^{i(\delta_2+1)\vartheta}+\text{c.c.},
\end{align}
where we have used the fact that
$$
e^{i(k+2)\vartheta}\left(\frac{1-te^{-2i\vartheta}}{1-te^{2 i \vartheta}}\right)=\pm 1
$$
is equal to its complex conjugate to simplify the expression for $W$. Adding all the three terms gives
\begin{align*}
X+Y+W=&\frac{(1-t)\cos(\delta_1\vartheta)}{2\sin\vartheta}-\frac{(1-t)}{4\sin\vartheta}\left(e^{i(\delta_2+2)\vartheta}\cdot\frac{1-te^{-2i\vartheta}}{1-te^{2i\vartheta}}+\text{c.c.}\right)\\
=&\frac{(1-t)\cos(\lambda_1-\lambda_2)\vartheta}{2\sin\vartheta}-\frac{(1-t)}{4\sin\vartheta}\left(e^{i(\lambda_1+\lambda_2+2)\vartheta}\cdot\frac{1-te^{-2i\vartheta}}{1-te^{2i\vartheta}}+\text{c.c.}\right),
\end{align*}
where we have again used the equation \eqref{BAESU2} to rewrite the second term. Now it can be straightforwardly checked that
$$
\frac{1}{(1-t)\sin\vartheta|1-te^{2i\vartheta}|^2}(A+B+C)=\Theta_{\lambda_1}(\vartheta)\Theta_{\lambda_2}(\vartheta),
$$
proving the proposition for even $k$ and odd $\lambda_1+\lambda_2$. The other cases are almost exactly the same. For example, when $\lambda_1+\lambda_2$ is even, the term coming from $\nu=0$ need to be treated specially, and $X$ in \eqref{ABCSum} is now
$$
X=t^{\delta_1/2}\sin\vartheta+ \sum_{\nu=2}^{\delta_1} t^{(\delta_1-\nu)/2}\cdot \left[\sin(\nu+1)\vartheta-t\sin(\nu-1)\vartheta\right].\\
$$
But one can easily verify that this still gives the expression for $X$ in \eqref{ABCResult}. Similarly, the term in $W$ coming from $\nu=k$ also need to be treated differently when $\lambda_1+\lambda_2+k$ is even. But one can check in all cases that \eqref{ABCResult} is always correct as long as $k>0$, proving the proposition.
\end{proof}

When $k=0$, although the vanishing theorem for higher cohomology groups no longer applies, the index formula still defines a 2D TQFT, with a one-dimensional Hilbert space $V$. The TQFT structure can be easily worked out. They are given by 
$$
\eta^{00}=1+t
$$
and
$$
f^{000}=1.
$$
There is only one deformed character given by 
$$
\Theta_0=\frac{1}{1+t},
$$
and
$w_\emptyset$ is now
$$
w_\emptyset=(1+t)w_0.
$$
The partition function on a genus-$g$ Riemann surface is given by
$$
d_g=(1+t)^{3-3g}.
$$

In the remainder of this appendix, we check
$$
d_{0,4}^{\lambda_1\lambda_2\lambda_3\lambda_4}=f^{\lambda_1\lambda_2\nu}f^{\lambda_3\lambda_4\rho}\eta_{\nu\rho}.
$$
From this expression, it is clear that the sum of $\lambda$'s has to be an even number in order for $d$ to be non-zero. In addition to $\delta_1$ and $\delta_2$, we also define
$$
\delta_3=|\lambda_3-\lambda_4|,\quad \delta_4=\mathrm{min}(\lambda_3+\lambda_4,2k-\lambda_3-\lambda_4),
$$
And it is convenient to define
\begin{align*}
L_1&=\mathrm{min}(\delta_1,\delta_3),\\
L_2&=\mathrm{max}(\delta_1,\delta_3),\\
L_3&=\mathrm{min}(\delta_2,\delta_4),\\
L_4&=\mathrm{max}(\delta_2,\delta_4).
\end{align*}
There are two different possibilities about values of $L_i$'s. We may have
$$
L_1\leq L_2\leq L_3\leq L_4
$$
or
$$
L_1\leq L_3\leq L_2\leq L_4,
$$
which respectively correspond to either the moduli space of parabolic bundles being non-empty or empty. We start with the first situation and assume that $\lambda_1+\lambda_2$ is even, then the $L$'s are all even. We also assume that $k$ is even. Now we break the summation into different parts
$$
d_{0,4}^{\lambda_1\lambda_2\lambda_3\lambda_4}=\sum_{\nu=0,2,4\ldots k}f^{\lambda_1\lambda_2\nu}f^{\lambda_3\lambda_4\nu}\eta_{\nu\nu}=A+B+C+D+E+F+G,
$$
where
$$
A=\frac{t^{(L_1+L_2)/2}}{1-t^2}
$$
is the contribution of $\nu=0$,
\begin{align*}
B&=\frac{1}{1-t}\sum_{\nu=2}^{L_1}t^{(L_1+L_2)/2-\nu},\\
C&=\frac{1}{1-t}\sum_{\nu=L_1+2}^{L_2}t^{(L_2-\nu)/2},\\
D&=\frac{(L_3-L_2)/2-1}{1-t}
\end{align*}
comes from $\nu=L_2+2,L_2+4,\ldots,L_3-2$,
\begin{align*}
E=&\frac{1}{1-t}\sum_{\nu=L_3}^{L_4-2}t^{(\nu-L_3)/2},\\
F=&\frac{1}{1-t}\sum_{\nu=L_4}^{k-2}t^{\nu-(L_3+L_4)/2},
\end{align*}
and finally
$$
G=\frac{t^{k-(L_3+L_4)/2}}{1-t^2}
$$
comes from $\nu=k$. After performing summation of various geometric series, we have:
\begin{align*}
B&=\frac{1}{1-t}\cdot\frac{t^{(L_2-L_1)/2}-t^{(L_2+L_1)/2}}{1-t^2},\\
C&=\frac{1-t^{L_2-L_1}}{(1-t)^2},\\
E&=\frac{1-t^{L_4-L_3}}{(1-t)^2},\\
F&=\frac{1}{1-t}\cdot\frac{t^{(L_4-L_3)/2}-t^{k-(L_3+L_4)/2}}{1-t^2}.
\end{align*}
Adding all the seven pieces up gives
\begin{multline}\label{EVerlinde3}
d_{0,4}=\frac{(L_3-L_2)/2+1}{1-t}+\frac{2t}{(1-t)^2}+\\ \frac{t^{(L_1+L_2)/2}+t^{(L_2-L_1)/2}+t^{k-(L_3+L_4)/2}+t^{(L_4-L_3)/2}}{(1-t^{-1})(1-t^2)}.
\end{multline}
Although we have assume that $k$ is even and $\lambda_1+\lambda_2$ is even, one can verify that equation (\ref{EVerlinde3}) is completely general. For example, if $k$ is odd and $\lambda_1+\lambda_2$ is even, then $A$, $B$, $C$, $D$, $E$ will be unchanged, while $G=0$ and
$$
F=\frac{1}{1-t}\sum_{\nu=L_4}^{k-1}t^{\nu-(L_3+L_4)/2}=\frac{1}{1-t}\cdot\frac{t^{(L_4-L_3)/2}-t^{k-(L_3+L_4)/2+1}}{1-t^2}.
$$
As a consequence, if we add all terms up, the coefficient of $t^{k-(L_3+L_4)/2}$ is still $\frac{1}{(1-t^2)(1-t^{-1})}$. Similarly, if $\lambda_1+\lambda_2$ is odd and $k$ is odd, then $A$ becomes zero but $B$ is changed to compensate for the vanishing of $A$, so that the final expression stays the same. If $\lambda_1+\lambda_2$ is odd while $k$ is even, $A$ and $G$ are both zero while $B$ and $F$ are changed and final expression still remains as (\ref{EVerlinde3}). 

When $t=0$, and $L_2\le L_3$
$$
\tilde{d}_{0,4}=(L_3-L_2)/2+1.
$$
So in this case \eqref{EVerlinde3} is exactly the same as \eqref{FourPuncture} with $h=0$ and 
\begin{align*}
h_1&=\frac{1}{2k}(2k-L_3-L_4)\\
h_2&=\frac{1}{2k}(L_4-L_3)\\
h_3&=\frac{1}{2k}(L_1+L_2)\\
h_4&=\frac{1}{2k}(L_2-L_1).
\end{align*}
The case with $L_2>L_3$ is very analogous, where various geometric series will start with $t^{h_0}$ instead of 1, with
$$
h_0=\frac{1}{2k}(L_2-L_3).
$$
Compared to the $L_2\leq L_3$ situation, only $C$, $D$ and $E$ in the seven-term decomposition are changed. They are now
\begin{align*}
C&=\frac{1}{1-t}\sum_{\nu=L_1+2}^{L_3}t^{(L_2-\nu)/2},\\
D&=t^{(L_2-L_3)}\frac{(L_2-L_3)/2-1}{1-t},\\
E&=\frac{1}{1-t}\sum_{\nu=L_2}^{L_4-2}t^{(\nu-L_3)/2}.
\end{align*}
So in this case we also have that
\be\label{EVerlinde5}
d_{0,4}=t^{kh_0}\left[\frac{(L_2-L_3)/2+1}{1-t}+\frac{2t}{(1-t)^2}\right]+\frac{\sum_{i=1,2,3,4}t^{kh_i}}{(1-t^{-1})(1-t^2)}.
\ee


\begin{thebibliography}{0}



 \bibitem[AU1]{AU1} J.E. Andersen and K. Ueno, \emph{Geometric construction of modular functors
from conformal field theory}, {Journal of Knot theory and its
Ramifications}, {\bf 16} 2 (2007) 127--202.

\bibitem[AU2]{AU2} J.E.~Andersen and K.~Ueno, \emph{Abelian Conformal Field theories and
Determinant Bundles}, {International Journal of Mathematics} {\bf
18} (2007) 919--993.

\bibitem[AU3]{AU3} J.E.~Andersen and K.~Ueno, \emph{Modular functors are determined by
their genus zero data}, {Quantum Topology}, {\bf 3} 3/4 (2012) 255--291.

\bibitem[AU4]{AU4} J.E.~Andersen and K.~Ueno, \emph{Construction of the Witten-Reshetikhin-Turaev TQFT from conformal field theory}, Invent. Math. {\bf 201 (2)} (2015) 519--559.



\bibitem[AK1]{AK1} {{J.E.~Andersen} and R.~{Kashaev},} \emph{A TQFT from Quantum Teichm\"uller Theory}, \emph{Commun. Math. Phys., }  {\bf 330} (2014) 887.

\bibitem[AK2]{AK2} {{J.E.~Andersen} and R.~{Kashaev},} \emph{A new formulation of the Teichm\"uller TQFT}, {\tt arXiv:1305.4291} (2013).

\bibitem[AK3]{AK3} {{J.E.~Andersen} and R.~{Kashaev},} \emph{Complex Quantum Chern-Simons}, {\tt arXiv:1409.1208} (2014).


\bibitem[AGO]{AGO}  J.~E.~Andersen, G.~Borot and N.~Orantin,
  \emph{Modular functors, cohomological field theories and topological recursion},
  {\tt arXiv:1509.01387} (2015).


\bibitem[AB]{AB} M.F.~Atiyah and R.~Bott, \emph{The Yang-Mills equations over Riemann surfaces}, {Philosophical Transactions of the Royal Society of London A: Mathematical, Physical and Engineering Sciences}, {\bf 308.1505} (1983) 523-615.

\bibitem[BL]{BL}  A. Beauville and Y. Laszlo, \emph{Conformal blocks and generalized theta functions}, Comm. Math. Phys. {\bf 164}.2 (1994) 385--419. 



\bibitem[BD]{BD} A.~Beilinson and V,~Drinfeld. {\it Quantization of Hitchin's integrable system and Hecke eigensheaves}. (1991) 15-19.


\bibitem[BS]{BS} A. Bertram and A. Szenes, \emph{Hilbert polynomials of moduli spaces of rank 2. Vector bundles. II.} Topology {\bf 32}.3 (1993) 599--609.

\bibitem[BR]{BR} I.~Biswas and S.~Ramanan. {\it An infinitesimal study of the moduli of Hitchin pairs}, { Journal of the London Mathematical Society} {\bf 49.2} (1994) 219-231.


\bibitem[BY]{Boden}
H.~U.~Boden and K.~Yokogawa,
\emph{Moduli Spaces of Parabolic Higgs Bundles and Parabolic K(D) Pairs over Smooth Curves: I},  International Journal of Mathematics {\bf 7.05} (1996) 573-598.

\bibitem[DW1]{DW1}G. Daskalopoulos and R. Wentworth, \emph{Local degeneration of the moduli space of vector bundles and factorization of rank two theta functions. I.} Math. Ann. {\bf 297}.3 (1993) 417--466. 

\bibitem[DW2]{DW2}G. Daskalopoulos and R. Wentworth, \emph{Factorization of rank two theta functions. II. Proof of the Verlinde formula.} Math. Ann. {\bf 304}.1 (1996) 21--51. 

\bibitem[D]{D} S. K. Donaldson,  \emph{Gluing techniques in the cohomology of moduli spaces}, Topological methods in modern mathematics (Stony Brook, NY, 1991), 137--170, Publish or Perish, Houston, TX, 1993. 

\bibitem[F]{F}  G. Faltings, \emph{A proof for the Verlinde formula}, J. Algebraic Geom. {\bf 3}.2 (1994) 347--374. 
    
\bibitem[FGT]{FGT} S.~Fishel,  I.~Grojnowski and C.~{Teleman} , {\it {The strong Macdonald conjecture and Hodge theory on the loop Grassmannian}}, {Annals of Math.}, {\bf 168} (2008) 175 -- 220.
		
\bibitem[FT]{FT} E,~Frenkel and C.~Teleman, \emph{Geometric Langlands Correspondence Near Opers}, {\tt arXiv:1306.0876} (2013).

		
\bibitem[GP]{GP} S.~Gukov and D.~Pei, {\it {Equivariant Verlinde formula from fivebranes and
  vortices}},  {\tt arXiv:1501.0131} (2015).
	
\bibitem[GPYY]{GPYY}
S.~Gukov, D.~Pei, W.~Yan, and K.~Ye, {\it {Equivariant Verlinde algebra from
  superconformal index and Argyres-Seiberg duality}},
  {{\tt arXiv:1605.0652}} (2016).

		
\bibitem[G]{G} R.~K.~Gupta, \emph{Characters and the q-analog of weight multiplicity}, Journal of the London Mathematical Society {\bf 2.1} (1987) 68-76.

  
\bibitem[M]{Mac}	I.G.~Macdonald, \emph{Orthogonal polynomials associated with root systems}, S\'em. Lothar. Combin, {\bf 45} (2000): B45a.

\bibitem[H1]{Hit} N.~Hitchin, \emph{The self-duality equations on a Riemann surface}, {Proc. London Math. Soc} {\bf 55.3} (1987): 59-126.
		
\bibitem[H2]{H} N.~Hitchin, \emph{Flat connections and geometric
    quantization}, Comm. Math. Phys., {\bf 131} (1990) 347--380.

\bibitem[HL]{HL} D.~Halpern-Leistner, \emph{The equivariant Verlinde formula on the moduli of Higgs bundles}, {\tt arXiv: 1608.01754} (2016).

\bibitem[KN]{KN} A.D.~King and P.E.~Newstead,  \emph{On the cohomology ring of the moduli space of rank 2 vector bundles on a curve}, Topology, {\bf 37}.2 (1998) 407-418, ISSN 0040-9383.

\bibitem[KNR]{KNR}S. Kumar, M.S. Narasimhan, A. Ramanathan, \emph{Infinite Grassmannians and Moduli
spaces of G-bundles}. Math. Ann., {\bf 300} (1994) 41--75. 

\bibitem[N]{N} P.~E.~Newstead, \emph{Characteristic classes of stable bundles of rank 2 over an algebraic curve}, Transactions of the American Mathematical Society {\bf 169} (1972) 337--345.


\bibitem[OY]{OY} S.~Okuda and Y.~Yoshida, \emph{G/G gauged WZW-matter model, Bethe Ansatz for q-boson model and Commutative Frobenius algebra}, {\tt arXiv:1308.4608} (2013).

\bibitem[P]{P} P.~Paradan, \emph{Formal geometric quantization II}, Pacific journal of mathematics 253.1 (2011): 169-211.

\bibitem[CP]{CP} C. Pauly, \emph{Espaces de modules de fibrés paraboliques et blocs conformes.} Duke Math. J. {\bf 84}.1 (1996) 217--235. 



\bibitem[R]{R} A.  Ramanathan, \emph{Stable principal bundles on a compact Riemann
			    surface},  { Math.  Anal.}  213 (1975) 129--152
			    
\bibitem[SR]{SR} S. Rayan, \emph{Geometry of co-Higgs bundles}, Oxford University D. Phil. Thesis, 2011.  


\bibitem[Sh]{Sh} S.~Shatz, \emph{Degeneration and specialization in algebraic families of vector bundles}, Compositio Math., {\bf 35}, (1976) 163-187.

\bibitem[Sz1]{Sz1} A. Szenes, \emph{Verification of Verlinde's formulas for SU(2)}. Internat. Math. Res. Notices {\bf 7} (1991) 93--98. 

\bibitem[Sz2]{Sz2} A. Szenes, \emph{Hilbert polynomials of moduli spaces of rank 2. Vector bundles. I.} Topology {\bf 32}.3 (1993) 587--597.


\bibitem[T1]{Te0} C.~Teleman, {\it Lie algebra cohomology and the fusion rules}, { Communications in mathematical physics} 173.2 (1995): 265-311.

\bibitem[T2]{Te1/2} C.~Teleman, {\it Verlinde factorization and Lie algebra chomology}, { Inventiones Mathematicae} {\bf 126} (1996) 249 -- 263.



\bibitem[T3]{Te1}
C.~{Teleman}, {\it {Borel-Weil-Bott theory on the moduli stack of G -bundles
  over a curve}},  { Inventiones Mathematicae} {\bf 134} (Sept., 1998)
  1--57.
	
\bibitem[T4]{Te3} C.~Teleman, {\it The quantization conjecture revisited}, { Annals of Mathematics} {\bf 152.1} (2000): 1-43.	

\bibitem[T5]{Te} C.~{Teleman}, {\it {K-theory of the moduli of bundles over a Riemann surface
  and deformations of the Verlinde algebra}},  { ArXiv Mathematics e-prints} {\tt  math/0306347}
  (June, 2003).

\bibitem[T6]{TeNote} C.~{Teleman}, {\it Loop Groups and G-bundles on curves},\\ {\tt https://math.berkeley.edu/\~{}teleman/seattle.pdf}, (unpublished).


\bibitem[TW]{TW} C.~{Teleman} and C.~T. {Woodward}, {\it {The Index Formula on the Moduli of
  G-bundles}}, { Annals of mathematics} (2009) 495-527.


\bibitem[Th]{Th}  M. Thaddeus, \emph{Stable pairs, linear systems and the Verlinde formula}, Invent. Math. {\bf 117}.2 (1994), 317--353.

\bibitem[TUY]{TUY} A. Tsuchiya, K. Ueno \& Y. Yamada, \emph{Conformal
    Field Theory on Universal Family of Stable Curves with Gauge
    Symmetries}, { Advanced Studies in Pure Mathmatics}, {\bf 19}
  (1989), 459--566.


\bibitem[V]{V}
E.~P. Verlinde, {\it {Fusion Rules and Modular Transformations in 2D Conformal
  Field Theory}},  { Nucl.Phys.} {\bf B300} (1988) 360.


\end{thebibliography}
\end{document}